\documentclass[a4paper,11pt]{article}
\usepackage{amssymb,amsmath,amsthm}
\usepackage{graphics,graphicx}
\usepackage[usenames, dvipsnames]{color}
\usepackage{srcltx}

\DeclareGraphicsExtensions{.eps} \textwidth=172mm
\definecolor{mygreen1}{cmyk}{0.1, 0.8, 0.8,0.4}
\definecolor{mygreen}{rgb}{0.1, 0.99, 0.478}
\definecolor{darkred}{rgb}{0.9,0.1,0.1}

\newtheorem{theorem}{Theorem}[section]
\newtheorem{proposition}[theorem]{Proposition}
\newtheorem{lemma}[theorem]{Lemma}
\newtheorem{corollary}[theorem]{Corollary}
\newtheorem{remark}[theorem]{Remark}

\title{Pointwise estimates for heat kernels of convolution type operators \thanks{The work was supported by SFB1283 of
German Research Council. The third and the forth authors were partially supported by the Russian Science Foundation, Project ¹ 14-50-00150} }
\author{A. Grigor'yan$^1$ , Yu. Kondratiev$^1$,  A. Piatnitski$^{2,3}$ and E. Zhizhina$^3$}

\begin{document}

\maketitle

\begin{center}\small
$^1$ {Fakult\"at f\"ur
 Mathematik, Universit\"at Bielefeld, 33501 Bielefeld, Germany}
\end{center}

\begin{center} \small
$^2$
The Arctic University of Norway,
Campus  Narvik,
Postbox 385, 8505 Narvik, Norway
\end{center}

\begin{center}\small
$^3$
Institute for Information Transmission Problems RAS, Moscow, 127051 Russia
\end{center}


\tableofcontents

\section{Introduction}

In this paper we are concerned with estimates of the heat kernel
(=fundamental solution) of certain evolution equations with non-local
elliptic part. The heat kernel of the classical heat equation
\begin{equation*}
\partial _{t}u-\Delta u=0,
\end{equation*}
where $\Delta $ is the Laplace operator in $\mathbb{R}^{d}$, is given by the
Gauss-Weierstrass function
\begin{equation}\label{lhe}
p_{t}\left( x,y\right) =\frac{1}{\left( 4\pi t\right) ^{d/2}}\exp \left( -%
\frac{\left\vert x-y\right\vert ^{2}}{4t}\right) .
\end{equation}%
For a more general parabolic equation
\begin{equation*}
\partial _{t}u-Lu=0,
\end{equation*}
where $L$ is a uniformly elliptic second order operator in divergence form,
Aronson \cite{Aron} proved the following Gaussian estimates for its heat
kernel:%
\begin{equation*}
p_{t}\left( x,y\right) \asymp \frac{C}{t^{d/2}}\exp \left( -\frac{\left\vert
x-y\right\vert ^{2}}{ct}\right) ,
\end{equation*}%
where the sign $\asymp $ means both $\leq $ and $\geq $ but with different
values of positive constants $C,c$.

A simplest heat equation with non-local elliptic part is
\begin{equation}
\partial _{t}u+\left( -\Delta \right) ^{\alpha /2}u=0,  \label{Dea}
\end{equation}%
where $0<\alpha <2$.  Applying the subordination techniques of \cite{Zol} to
the Gauss-Weierstrass function, one obtains that the heat kernel of (\ref%
{Dea}) satisfies the following estimates%
\begin{equation}
p_{t}\left( x,y\right) \asymp \frac{C}{t^{d/\alpha }}\left( 1+\frac{%
\left\vert x-y\right\vert }{t^{1/\alpha }}\right) ^{-\left( d+\alpha \right)
}  \label{stable}
\end{equation}%
(see also \cite{Bend}). Note that $\left( -\Delta \right) ^{\alpha /2}$ is
an integro-differential operator of the form%
\begin{equation}
\left( -\Delta \right) ^{\alpha /2}f\left( x\right) =c_{d,\alpha }\,\mathrm{%
p.v.}\int_{\mathbb{R}^{d}}\frac{f\left( x\right) -f\left( y\right) }{%
\left\vert x-y\right\vert ^{d+\alpha }}dy.  \label{deal}
\end{equation}%
The heavy tail of the heat kernel in the estimate (\ref{stable}) is a
consequence of the heavy integral kernel in (\ref{deal}). Similar estimates hold also
for non-local heat kernels on fractals \cite{Gri}.

A natural class of non-local operators arises on graphs. Let $\Gamma $ be a
countable, locally finite, connected graph. Let $d\left( x,y\right) $ be the
graph distance on $\Gamma $. The discrete Laplace operator  $\Delta $ on $%
\Gamma $ acts on functions $f:\Gamma \rightarrow \mathbb{R}$ as follows:%
\begin{equation*}
\Delta f\left( x\right) =\frac{1}{\deg \left( x\right) }\sum_{\{y\in \Gamma
:y\sim x\}}\left( f\left( y\right) -f\left( x\right) \right) =\sum_{y\in
\Gamma }\left( f\left( y\right) -f\left( x\right) \right) J\left( x,y\right)
,
\end{equation*}%
where
\begin{equation*}
J\left( x,y\right) =\frac{1}{\deg \left( x\right) }\mathbf{1}_{\left\{
d\left( x,y\right) =1\right\} }.
\end{equation*}%
Davies has obtained in \cite{Davies} the upper bounds of the heat kernel $%
p_{t}\left( x,y\right) $ of the heat equation $\partial _{t}u-\Delta u=0$ on
$\Gamma $ that in the case of uniformly bounded degree $\deg \left( x\right)
$ of vertices amounts to
\begin{equation}
p_{t}\left( x,y\right) \leq \exp \left( -ct\Phi \left( \frac{d\left(
x,y\right) }{ct}\right) \right) ,  \label{Dav}
\end{equation}%
where%
\begin{equation*}
\Phi \left( \xi \right) =\sup_{\lambda >0}\left\{ \xi \lambda -\cosh \lambda
\right\} =\xi \ln \left( \xi +\sqrt{\xi ^{2}+1}\right) -\sqrt{1+\xi ^{2}}.
\end{equation*}%
Since
\begin{equation}\label{davies}
\Phi \left( \xi \right) \sim \frac{\xi ^{2}}{2}\ \ \text{as\ }\xi
\rightarrow 0\ \ \ \text{and\ }\ \ \Phi \left( \xi \right) \sim \xi \ln \xi
\ \ \text{as }\xi \rightarrow \infty ,
\end{equation}%
the estimate (\ref{Dav}) implies for small $\frac{d\left( x,y\right) }{t}$
the Gaussian estimate%
\begin{equation*}
p_{t}\left( x,y\right) \leq \exp \left( -\frac{d^{2}\left( x,y\right) }{ct}%
\right) ,
\end{equation*}%
and for large $\frac{d\left( x,y\right) }{t}$
\begin{equation*}
p_{t}\left( x,y\right) \leq \exp \left( -cd\left( x,y\right) \ln \frac{%
d\left( x,y\right) }{ct}\right) .
\end{equation*}%
The estimate (\ref{Dav}) gives a rather sharp upper bound of the tail of the
heat kernel on an arbitrary graph because on $\Gamma =\mathbb{Z}$ the heat
kernel admits the following two-sided estimate%
\begin{equation*}
p_{t}\left( x,y\right) \asymp \frac{C}{\left( t+d\left( x,y\right) \right)
^{1/2}}\left( -2t\Phi \left( \frac{d\left( x,y\right) }{2t}\right) \right)
\end{equation*}%
(see \cite{Pang}).

In this paper we consider the non-local operator $A$ on functions $f:\mathbb{%
R}^{d}\rightarrow \mathbb{R}$ given by%
\begin{equation}\label{gene}
Af=a\ast f-f,
\end{equation}%
where the convolution kernel $a$ is such that
\begin{equation}\label{a1}
a(x) \ge 0; \quad a(x) = a(-x); \qquad  a(x) \in L^{\infty}(\mathbb R^d) \cap L^1(\mathbb R^d),
\end{equation}
\begin{equation}\label{a2}
\int_{\mathbb R^d} a(x) dx =1, \quad \int_{\mathbb R^d} |x|^2 a(x) dx  < \infty.
\end{equation}
In particular, under condition \eqref{a2} there exists a positive definite matrix $\sigma = \{ \sigma_{i j} \}$ with $\sigma_{i j} = \int_{\mathbb R^d} x_i x_j a(x) dx$.
The third condition in \eqref{a1} implies that $a(x) \in L^2(\mathbb R^d)$, and for the Fourier transform $\hat a(p)$ we have
\begin{equation}\label{ahat}
\hat a(p) \in C_b(\mathbb R^d) \cap L^2(\mathbb R^d), \quad \max_{\mathbb R^d} \hat a(p) = \hat a(0) =1, \quad \hat a(p) \to 0 \;
\mbox{ as } \; |p| \to \infty.
\end{equation}
The operator $A$ takes a form of an integro-differential operator as
follows:%
\begin{equation*}
Af\left( x\right) =\int_{\mathbb{R}^{d}}\left( f\left( y\right) -f\left(
x\right) \right) a\left( x-y\right) dy.
\end{equation*}%
An essential difference from the operator (\ref{deal}) is that the integral
kernel $a\left( x-y\right) $ of $A$ is bounded and integrable. Surprisingly,
these assumptions do not make the task of estimating of the heat kernel
easier.

Since $A$ is a bounded operator in $L^{2}\left( \mathbb{R}^{d}\right) $, its
heat semigroup $e^{tA}$ can be easily computed by using the exponential
series that leads to%
\begin{equation*}
e^{tA}=e^{-t}e^{ta\ast }=e^{-t}\sum_{k=0}^{\infty
}t^{k}\frac{a^{\ast k}}{k!}=e^{-t}\mathrm{Id}+e^{-t}\sum_{k=1}^{\infty }t^{k}%
\frac{a^{\ast k}}{k!},
\end{equation*}%
By removing the singular part $e^{-t}\mathrm{Id}$ of the heat semigroup, we
obtain the \emph{regularized} heat kernel%
\begin{equation}\label{v}
v\left( x,t\right) =e^{-t}\sum_{k=1}^{\infty }t^{k}\frac{a^{\ast k}\left(
x\right) }{k!}
\end{equation}%
with the source at the origin.
In other words, for any $f\in L^{2}\left( \mathbb{R}^{d}\right) $, a solution to the non-local Cauchy problem
\begin{equation}\label{nlcau}
  \begin{array}{c}
    \partial_t u-Au=0, \\[1mm]
    \displaystyle
    u\big|_{t=0}=f
  \end{array}
\end{equation}
has the form $u(x,t)=e^{-t}f(x)+(v\ast f)(x,t)$ with $v$ given by \eqref{v}.
In particular, the fundamental solution of the problem (\ref{nlcau}) is
\[
u\left( x,t\right) =e^{-t}\delta \left( x\right) +v\left( x,t\right) .
\]
 The function $v$ is the main subject of this paper.

A probabilistic interpretation of the function $v(x,t)$ is of great interest. Under conditions \eqref{a1}, \eqref{a2} the operator $A$ defined in \eqref{gene} is a generator of   a continuous time Markov jump process. If this process starts at zero,
its transition probability has a regular  part and a singularity at zero,
and $v(x,t)$ is the density of the regular part. The results of this work allow us to describe the large time behaviour of this Markov
process in different regions of the space.  In particular we obtain the local moderate and large deviations results for this Markov process.

Recent years there is an essential progress in studying the large time behaviour of solutions to evolution problems in $\mathbb R^d$
for convolution type operators with integrable kernels, see, for instance, \cite{AMRT}, \cite{CCR}, \cite{CFRT}, and the references therein.
One of the key questions of interest here is obtaining point-wise estimates for the corresponding nonlocal heat kernels and solutions. To our best knowledge there are just few papers devoted to this topic. In \cite{BCF} the asymptotic behaviour of fundamental solution  for evolution equations with a convolution kernel has been considered.  For Gaussian and compactly supported kernels that are radially symmetric, two-sided estimates  have been obtained. Since \cite{BCF} mostly deals with problems with unbounded initial conditions,
the authors focuses  on the behaviour of heat kernel  in the region of extra large  $|x|\gg t$, and their estimates are rather loose in other regions. The kernels showing sub-exponential decay at infinity have been studied in \cite{FiTk},  this work deals with the asymptotic behaviour of the fundamental solution in the region $|x|\gg t$.

Closely related results on point-wise estimates for a resolvent kernel of non-local convolution type operators have been obtained in the
recent work \cite{komopizhi}. Both polynomially and exponentially decaying kernels were considered. With the help of these estimates
point-wise bounds for the principal eigenfunction of non-local Schr\"odinger operator were deduced.


\smallskip
In the present paper we deal with convolution kernels $a(x)$ that decay at infinity at least exponentially and  admit an estimate from above by a radially symmetric function:  $a(x)\leq ce^{-b|x|^p}$ with $b>0$ and $p\geq 1$.

The large time behaviour of the studied heat kernel depends crucially on the relation between $|x|$ and $t$. We consider separately
four different regions in $(x,t)$ space, namely,\\ (i) $|x|=O(t^{1/2})$,\\  (ii) $t^\frac12\ll|x|\ll t$,\\ (iii) $|x|\sim t$,\\
 (iv) $|x|\gg t$.\\
 In particular, it will be shown that in the region (iv) the function \ $-\ln v(x,t)$ behaves like $|x|\Big(\ln\frac{|x|}t\Big)^\frac{p-1}p$
 for $a(x)\sim e^{-b|x|^p}$ with $p\geq 1$, and like $|x|\ln\frac{|x|}t$ for $a(x)$ with a finite support.

 Remark that for the corresponding Markov jump process with the generator defined in \eqref{gene} the region (i) corresponds to the standard deviations where the local central
 limit theorem applies, (ii) is the region of the moderate deviations, (iii)  is the region of large deviation, and (iv) should probably
 be called the "extra large" deviation region.

Before considering the case of  generic convolution kernels with a light tail we first study the  Gaussian kernels for which the
k-th convolution admits an explicit formula. This allows us to find the asymptotics of the corresponding heat kernel in all the regions
mentioned above, see Theorem \ref{Gauss}.

The Gaussian asymptotics of a generic non-local heat kernel in the region (i) is a consequence of the (local) central limit theorem.
It is interesting to observe that in the region (ii) the logarithmic asymptotics of the non-local heat kernel still remains the same as
for the classical heat kernel with the covariance matrix $\sigma$, see Theorem \ref{LTail-1var2}.
The transition between Gaussian and non-Gaussian behaviour occurs in the region $x=rt$. For small $r$ the behaviour
is still close to Gaussian, while as $r\to\infty$ the asymptotics of the non-local heat kernel does not look like Gaussian at all,
as shown in Theorems \ref{LTail-2_1D} and \ref{LTail-2_MD}.
The difference is getting even more drastic in the region $|x|\gg t$, see Theorem \ref{LTail-1var3}.


\section{Gaussian convolution kernel}
\setcounter{equation}{0}

We consider in this section the case of a Gaussian convolution kernel:
\begin{equation}\label{ga}
a(x) \ = \
\frac{1}{(4 \pi)^{d/2}} \ e^{-\frac{x^2}{4}} , \quad \hat a(p) \ = \ e^{-p^2}.
\end{equation}
In this case the convolutions $a^{*k}(x)$ admit explicit formulae for all $k\geq 1$
which essentially simplify our analysis. The large time asymptotics (or log asymptotics) of the fundamental solution depends essentially on
the relation between $x$  and $t$. We consider separately four different regions in $(x,t)$-space, namely,   $|x|= O( t^{\frac12})$ and $|x| \sim t^{\frac{1+\delta}{2}}$ with  $0<\delta<1$, or $\delta=1$,  or $\delta>1$.

Denote $\Phi_G(r)=1+2\xi_r\ln \xi_r-\xi_r$, where $\xi_r$ is a solution to the equation $\xi^2\ln\xi=\frac{r^2}4$.

\begin{theorem}[Gaussian kernel]\label{Gauss}
Let the convolution kernel $a(x)$ be defined by \eqref{ga}. Then for the function $v(x,t)$ defined by \eqref{v} the following asymptotics holds as $t \to \infty$  (see Figure {\rm \ref{f_gauss}}): \\
1) For any $r>0$, if $|x| \le r t^{\frac12}$, then
\begin{equation}\label{GaussSD}
v(x,t) = \frac{1}{(4 \pi t)^{d/2}} e^{-\frac{x^2}{4t}} (1+ o(t^{-\frac14})).
\end{equation}
2) For any $r>0$, if $|x| = r t^{\frac{1+\delta}{2}}$ with  $0<\delta<1$, then
\begin{equation}\label{GaussMD}
\frac{\ln v(x,t)}{\frac{x^2}{4t}} \to -1.
\end{equation}
In particular, if $\; r_1 t^{\frac{1+\delta}{2}} \le |x| \le r_2 t^{\frac{1+\delta}{2}}$ with some $0< r_1 < r_2$ and  $0<\delta<1$, then
$$
 e^{- \frac{r_2^2}{4} t^{\delta}(1 + o(1)) } \ \le \ v(x,t) \ \le \  e^{- \frac{r_1^2}{4} t^{\delta}(1 + o(1)).}
$$
3)   For any $r>0$, if $|x| = r t$, then 
\begin{equation}\label{GaussLD}
\frac{\ln v(x,t)}{t} \to - \Phi_G (r).
\end{equation}
Furthermore, the function   $\Phi_G (r)$ possesses the following properties:
$$
\begin{array}{cl}
 0<\Phi_G (r)<r^2/4  \ \ \ & \mbox{ for all } \ r\not=0,  \\[2mm]
  \Phi_G (r) = \frac{r^2}{4}(1+o(1)) \ \ \ &  \mbox{ as } \ r \to 0+\\[2mm]
  \Phi_G (r)=  r \sqrt{\ln r} (1+ o(1)) \ \ \ & \mbox{ as } \  r \to \infty
\end{array}
$$
4) If $|x| > t^{\frac{1+\delta}{2}}$ with $\delta>1$, then
\begin{equation}\label{GaussSLD}
\frac{\ln v(x,t)}{ |x| \sqrt{\ln \frac{|x|}{t}}} \to  -1.
\end{equation}

\end{theorem}
\begin{figure}[h!]
\centerline{\includegraphics [width=7in]{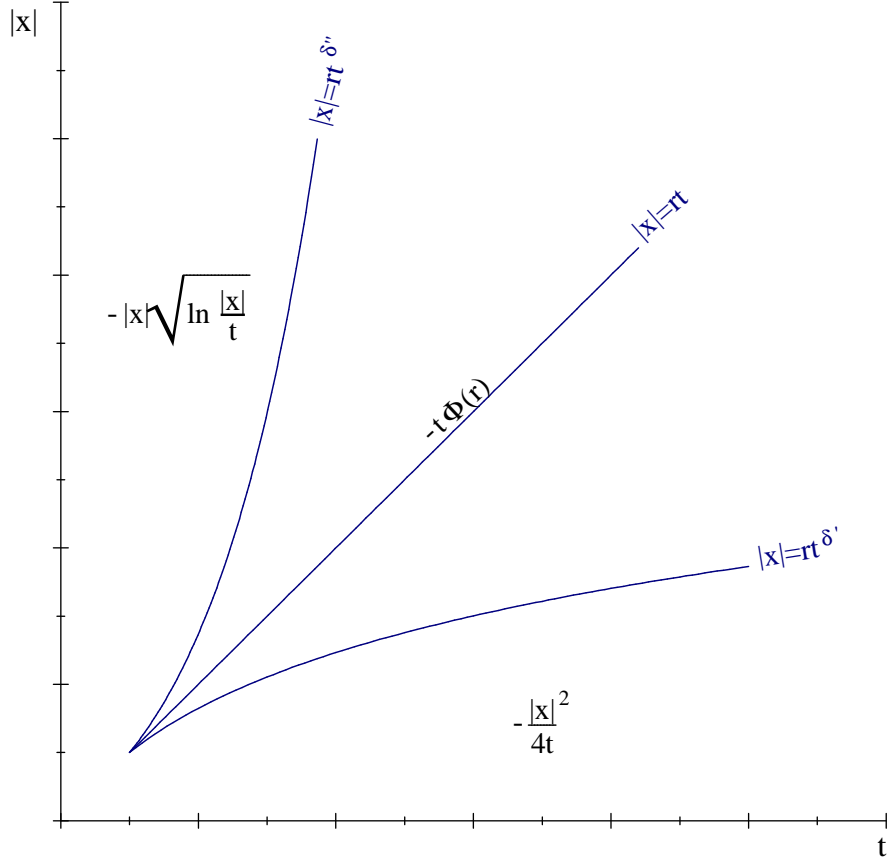}}
\caption{The large time behaviour of the function $\ln v(x,t)$ depends crucially on whether
$|x|\ll t$ (under the lower curve), or $|x|\sim t$ (the middle curve), or  $|x|\gg t$ (over the upper curve).
Here $\delta'<1$ and $\delta''>1$.}
\label{f_gauss}
\end{figure}
\begin{corollary}
 For any $r>0$, if $|x| = r t^{\frac{1+\delta}{2}}$ with $\delta>1$, then it follows from \eqref{GaussSLD} that
$$
\frac{\ln v(x,t)}{t^{\frac{\delta+1}{2}} \sqrt{\ln t}} \to - \tilde c(\delta, r), \quad \mbox{with} \quad \tilde c(\delta, r)= r \sqrt{\frac{\delta-1}{2}}.
$$
\end{corollary}

\begin{remark}\label{rem_2} {\rm In the case $\delta>1$
the function $
\frac{x^2}{4 t} \ = \frac{r^2}{4} t^\delta$  exhibits the faster polynomial growth at infinity  than the function $
t^{\frac{\delta+1}{2}} \sqrt{ \ln t}$.
Consequently,  in this region the nonlocal heat kernel $v(x,t)$  has a more ''fat''\ tail
$v(x,t) \sim e^{ - \tilde c(\delta,r) t^{\frac{\delta+1}{2}} \sqrt{ \ln t}}$ than the classical heat kernel $w(x,t) = \frac{1}{(4 \pi)^{d/2} t^{d/2}} e^{ -\frac{x^2}{4t}} \sim e^{- \frac{r^2}{4} t^\delta}$.
}
\end{remark}

In the next sections we prove all statements of Theorem \ref{Gauss}.

\subsection{Asymptotics in the case $|x| \le r t^{\frac12}$}

The asymptotics \eqref{GaussSD} follows from the local limit theorem for a general probability distribution that satisfies \eqref{a1} - \eqref{a2}. To justify the estimate for the reminder  in \eqref{GaussSD} we give a short analytic proof based on the following representation for $v(x,t)$
$$
v(x,t) \ = \  \int_{\mathbb R^d} e^{ixp} \Big( e^{-t(1-\hat a(p))} - e^{-t} \Big) dp.
$$
This integral can be rewritten as the following sum:
\begin{equation}\label{vsum}
v(x,t) \ = \ \int\limits_{|p| < \sqrt{ \ln 2t}} e^{ixp} e^{-t(1-\hat a(p))} dp  - e^{-t} \int\limits_{|p| < \sqrt{ \ln 2t}} e^{ixp} dp +
\int\limits_{|p| > \sqrt{ \ln 2t} } e^{ixp} \Big( e^{-t(1-\hat a(p))}  - e^{-t} \Big) dp.
\end{equation}


The second and the third integral in \eqref{vsum} can be estimated from above by $O(e^{-t} (\ln t)^{\frac{d}{2}})$ and $o(e^{-t} t)$ correspondingly.
Denoting  $1-\hat a(p)=p^2-p^4 f(p)$  and taking into account the relation $p^4f(p)=O(-t^{\frac43})$ valid for $|p|<t^{-\frac13}$,  for the first integral in \eqref{vsum}  we get
\begin{equation}\label{I1bis}
\int\limits_{|p| < \sqrt{\ln 2t}} e^{ixp} e^{-t(1-\hat a(p))} dp =  \int\limits_{|p|<t^{- \frac13}} e^{ixp}  e^{-t (p^2- p^4
f(p) )} dp + \int\limits_{ t^{- \frac13} < |p| < \sqrt{\ln 2 t}} e^{ixp} e^{-t(1-\hat a(p))} dp
\end{equation}
$$
= \
\frac{1}{(4 \pi)^{d/2} \ t^{\frac{d}{2}}} \ e^{-\frac{x^2}{4t}} \left( 1 +  o (t^{-\frac14}) \right) + o(e^{-t^\frac14}).
$$
This yields \eqref{GaussSD}.

\subsection{The case  $|x| \ = \ r \ t^{\frac{1+\delta}{2}}, \  0<\delta \le 1$}

In this region we exploit the first  representation for $v(x,t)$ in  (\ref{v}).
Since
\begin{equation}\label{aastk}
a^{\ast k} (x) \ = \ c_k \ e^{-\frac{x^2}{4 k}} \quad \mbox{ with } \; c_k = \frac{\tilde c}{k^{d/2}}, \; \tilde c >0,
\end{equation}
then using Stirling's approximation we get 
\begin{equation}\label{term}
 \frac{ t^k  a^{\ast k} (x)}{k!} \
= \ \exp \left\{ k \ln t - k \ln k + k - \frac{x^2}{4 k} - c(d) \ln k +r_k\right\}
\end{equation}
with a constant $c(d) = \frac{d+1}{2}$ and $|r_k|\leq C$. 
Let us estimate the maximal term in the sum
\begin{equation}\label{v-sum}
\sum_{k=1}^{\infty} \frac{ t^k  a^{\ast k} (x)}{k!}.
\end{equation}
To this end we introduce a function
\begin{equation}\label{S}
S(z,t) \ = \ \big(z \ln t - z \ln z + z - \frac{x^2}{4 z} - c(d) \ln z\big)\Big|_{x^2=rt^{1+\delta}}, \quad z > 0,
\end{equation}
and locate $\max\limits_{z>0} S(z,t)$ in $z$ for each $t>0$. 
Since for each positive $t$ the function $S(z,t)$ tends to $ - \infty$ both as $z\to 0 $ and as $z \to \infty$,  it attains its maximum
on $(0,+\infty)$. Denote
$$
\hat z \ = \ \hat z(t) \ = \ argmax_{z>0} S(z,t).
$$

\begin{proposition}\label{P1}
Let $\delta>0$, then $\hat z(t) \ = \ t \hat \xi (t)$, where $\hat \xi = \hat \xi(t) $ is the solution of equation
\begin{equation}\label{xi_bis}
\frac{4}{r^2} \xi^2 \ln \xi + \frac{4 c(d)}{r^2 t} \xi \ = \ t^{\delta-1}.
\end{equation}
Moreover, $\hat \xi (r,t) =  \xi_r(t) (1+o(1)), \; t \to \infty, $ where $\xi_r (t)$ is the solution of equation
\begin{equation}\label{xi-main}
\frac{4}{r^2} \xi^2 \ln \xi  \ = \ t^{\delta-1}.
\end{equation}
\end{proposition}

\begin{proof}
The  maximum point of  $S(z,t)$ is defined by the equation
\begin{equation}\label{xiz}
\frac{\partial}{\partial z} S(z,t) \ = \ \ln t - \ln z + \frac{x^2}{4 z^2} - \frac{c(d)}{z} \ = \ 0.
\end{equation}
Making the change of variables $z = t \xi$, we rewrite \eqref{xiz} as \eqref{xi_bis}.
Denote a solution of this equation by $\hat\xi (r, t)$. In what follows if it does not lead to ambiguity we drop the arguments of the function $\hat\xi (r, t)$. Observe that $\hat\xi > 1$ for sufficiently large $t$. Indeed, for $\xi\in(0,1]$
we have $\xi^2\ln\xi\leq 0$, and $\frac{\xi}{t} = o\big(t^{\delta-1}\big)$ as $t\to\infty$. This yields the required
inequality.

Notice that the function on the left-hand side of \eqref{xi_bis} is increasing as $\xi \in (1, +\infty)$ and therefore, equation  \eqref{xi_bis} has a
unique solution for large $t$.
It is easy to see that
$$
\frac{4 c(d)}{r^2 t} \xi=o\Big(\xi^2 \ln \xi\Big)\qquad\hbox{if }\xi\to\infty \; \hbox{and } \; t\geq 1,
$$
and
$$
\frac{4 c(d)}{r^2 t} \xi=o(t^{\delta-1}) \quad \mbox{ as } \; t\to\infty \; \hbox{and } \; \xi \; \hbox{is bounded. }
$$
Consequently, the solution $\hat\xi(r,t)$ of \eqref{xi_bis} can be approximated for large $t$ by the solution $\xi_r(t)$ of equation \eqref{xi-main}.
\end{proof}
\medskip

We consider separately the following cases: $0<\delta<1$ and $ \delta=1$. \\[2mm]
In the case $0<\delta<1$ we have  $ t^{\delta-1} \to 0$ as $t \to \infty$ and, therefore,
the solution of \eqref{xi-main} converges to $1$. The Taylor expansion of $\frac4{r^2} \xi^2 \ln \xi$ about
$1$ reads
$$
\frac4{r^2} \xi^2 \ln \xi=\frac4{r^2} (\xi - 1)+\frac6{r^2} (\xi - 1)^2+O( (\xi - 1)^3), \quad \xi-1\to0.
$$
Combining this expansion with \eqref{xi_bis} we obtain
$$
\frac4{r^2} (\xi - 1)+\frac6{r^2} (\xi - 1)^2=t^{\delta-1}-\frac{4 c(d)}{r^2 t}+
O((\xi-1)^3+(\xi-1)t^{-1}).
$$
The straightforward computations yield
$$
\hat\xi (r,t) \ = \ 1+ \frac{r^2}{4} t^{\delta-1}-\Big(\frac{3r^2}8t^{2(\delta-1)}+
 \frac{c(d)}t\Big)+ {\rm o}\big(\max\{t^{2(\delta-1)}, t^{-1}\}\big)
$$
and
$$
\hat z \ = \ t \hat\xi \ = \ t + \frac{
r^2}{4}  t^\delta -\Big(\frac{3r^2}8t^{2\delta-1}+
 c(d)\Big)+ {\rm o}\big(\max\{t^{2\delta-1}, 1\}\big).
$$
Substituting this expression for $\hat z$ in (\ref{S}) and considering the relation $x^2=r^2 t^{1+\delta}$ we get
\begin{equation}\label{maxS_bis}
S(\hat z, t) \  = \ t - \frac{r^2}{4} t^\delta + \frac{r^4}{16}t^{2\delta-1}
-c(d)\ln t+o\big(\max\{t^{2\delta-1},\ln t\}\big).
\end{equation}
Now from (\ref{v}) and (\ref{term}), taking into account the fact that $a^{\ast k}(x)>0$ for all $k$ and $x$, we obtain the following
estimate of $v(x,t)$ from below:
\begin{equation}\label{v1b}
v(x,t) \ = \ e^{-t}  \sum_{k=1}^{\infty} \frac{ t^k  a^{\ast k} (x)}{k!}  \ \ge \   e^{-t+ S(\hat z,t)}  \ = \  e^{- \frac{r^2}{4}
t^{\delta}+ \frac{r^4}{16}t^{2\delta-1}
-c(d)\ln t+o(\max\{t^{2\delta-1},\,\ln t\})},\quad\hbox{as }t\to\infty.
\end{equation}

To get an upper bound on $v(x,t)$ we divide the sum in \eqref{v} into two parts, in the first sum the summation index varies from $1$
to $n_0$ where $n_0$ is chosen in such a way that
\begin{equation}\label{n}
\frac{t^{n+1} e^{-\frac{x^2}{n+1}}}{(n+1)!} \ \frac{n!}{t^{n} e^{-\frac{x^2}{n}}} \ = \  \frac{t}{n+1} \ e^{\frac{x^2}{n(n+1)}} \ < \
\frac12 \quad \mbox{for all } \; n \ge n_0.
\end{equation}
Using the relation  $x^2 = r^2 t^{1+\delta}$ and the fact that $f(u) = u e^{cu^2}$ is an increasing  function for any $c>0$, we have
$$
 \frac{t}{n+1} \ e^{\frac{x^2}{n(n+1)}} < \frac13 e^{\frac{r^2}{9} t^{\delta-1}} < \frac12  \quad \mbox{for all } \; n > n_0=[3t] \;
 \mbox{ and } \; t \geq \Big(\frac {r^2}{9\ln(3/2)}\Big)^{\frac1{1-\delta}}.
$$
This implies that
\begin{equation}\label{g3t}
\sum_{n > 3t} \frac{t^n \ a^{\ast n}(x)}{n!} \ \le \ e^{S(\hat z, t)} \ = \  e^{t - \frac{c}{4}  t^\delta +
\frac{c^2}{16}t^{2\delta-1}
-c(d)\ln t+o(\max\{t^{2\delta-1},\,\ln t\})}, \quad t \to \infty.
\end{equation}
Due to \eqref{term} and \eqref{S} the upper bound for the sum $\sum_{n =1}^{[3t]} \frac{t^n \ a^{\ast n}(x)}{n!}$ reads
\begin{equation}\label{l3t}
\sum_{n =1}^{[3t]} \frac{t^n \ a^{\ast n}(x)}{n!}\leq3Ct e^{S(\hat z,t)}=
 e^{t - \frac{r^2}{4}  t^\delta + \frac{r^4}{16}t^{2\delta-1}
-(c(d)-1)\ln t+o(\max\{t^{2\delta-1},\,\ln t\})}, \quad t \to \infty.
\end{equation}
From (\ref{g3t}) and \eqref{l3t}
we derive the estimate of $ v(x,t)$ from above as $t \to \infty$:
\begin{equation}\label{v1a}
\begin{array}{c}
\displaystyle
v(x,t) \ =  \ e^{-t}  \sum_{k=1}^{[3t]} \frac{ t^k  a^{\ast k} (x)}{k!} +  e^{-t}  \sum_{k > 3t} \frac{ t^k  a^{\ast k} (x)}{k!}
\\[5mm]
\displaystyle
 \le \  e^{- \frac{r^2}{4} t^{\delta} + \frac{r^4}{16}t^{2\delta-1}
-(c(d)-1)\ln t+o(\max\{t^{2\delta-1},\,\ln t\})}.
\end{array}
\end{equation}
Finally from \eqref{v1b} and \eqref{v1a} we get \eqref{GaussMD}.

\begin{remark}\label{rem_1} {\rm Since $\frac{x^2}{4 t} = \frac{r^2 t^\delta}{4}, \; 0<\delta<1$,  the logarithmic
asymptotics of $v(x,t)$ coincides with that for the classic heat kernel:
\begin{equation}\label{R1}
\frac{\ln v(x,t)}{\frac{x^2}{4 t}} \ \to \ -1 \quad \mbox{ as } \quad t\to \infty.
\end{equation}
Moreover, for $\delta\leq \frac12$ estimates \eqref{v1b} and \eqref{v1a} take the form
\begin{equation}\label{f_rem1}
C_1\, t^{\frac{1-d}2} e^{- \frac{r^2}{4} t^{\delta}
}\leq
v(x,t) \le \ C_2\,t^{\frac{1-d}2} e^{- \frac{r^2}{4} t^{\delta} }
\end{equation}
with $C_1,\,C_2>0$.
For $\delta\in(\frac12,1)$ estimates \eqref{v1b} and \eqref{v1a} imply that
\begin{equation}\label{f_rem1_2}
v(x,t)
 = \  e^{- \frac{r^2}{4} t^{\delta} + \frac{r^4}{16}t^{2\delta-1}
+o(t^{2\delta-1})}.
\end{equation}
}
\end{remark}


\bigskip
We proceed with the case $\delta=1$. In this case equation \eqref{xi-main} reads
\begin{equation}\label{c2}
\xi^2 \ln \xi  \ = \ \frac{r^2}{4}.
\end{equation}
It is easy to check that equation \eqref{c2} has a unique solution   $\xi_r \in (1, \infty)$. Then for solution  $\hat\xi (r, t)$ of
 \eqref{xi_bis} by the implicit function theorem it follows that $\hat\xi (r, t) = \xi_r + O(t^ {-1})$. Therefore,
$$
S(\hat z, t) \ = \ t \xi_r \ln t - t \xi_r (\ln t + \ln \xi_r) + t \xi_r - t \xi_r \ln \xi_r
+O(\ln t)\ = \ t(\xi_r - 2 \xi_r \ln \xi_r)+O(\ln t), \; \hat z = t \hat \xi,
$$
and using the same arguments as above we have
\begin{equation}\label{v3}
  v(x,t) \ = \ e^{-t(1+ 2 \xi_r \ln \xi_r - \xi_r) + O(\ln t) }, \quad t \to \infty,
\end{equation}
where $\xi_r>1$ is the solution of (\ref{c2}). Thus the logarithmic asymptotics of $v(x,t)$
is given by
\begin{equation}\label{R3}
\frac{\ln v(x,t)}{t} \ \to \ - \Phi_G (r), \quad \mbox{ as } \; t \to \infty,
\end{equation}
where $\Phi_G (r) = \Phi_G (\xi_r) = 1 + 2 \xi_r \ln \xi_r - \xi_r$.

\begin{lemma}
For any $r>0$
\begin{equation}\label{R4}
  0 \ < \ \Phi_G (r) \ < \ \frac{r^2}{4}.
\end{equation}
Moreover,
\begin{equation}\label{predely}
\begin{array}{c}
\displaystyle
\Phi_G (r) = \frac14 r^2 (1+o(1)), \quad \mbox{\rm as } \; r \to 0, \\[3mm]
\displaystyle
\Phi_G (r) = r \sqrt{\ln r} (1+o(1)) \ \ \ \hbox{\rm as } r \to \infty.
\end{array}
\end{equation}
\end{lemma}

\begin{proof}
Let  $\gamma(\xi)=1+2 \xi \ln \xi - \xi$.
To prove the lower bound in (\ref{R4}) we notice that $\gamma (1) = 0$ and $\frac{\partial}{\partial \xi} \gamma (\xi) \ = \ 1+2 \ln
\xi \ge 1$ (for $\xi \ge 1$). To prove that
$$
1 + 2 \xi_r \ln \xi_r - \xi_r \ < \ \frac{r^2}{4} \ = \ \xi_r^2 \ln \xi_r,
$$
we denote by $\varkappa(\xi) = \xi^2 \ln \xi$.
Then $\gamma (1) = \varkappa (1) = 0$, and $\gamma'(\xi)< \varkappa' (\xi) , \; \xi>1$. This yields the desired upper bound in (\ref{R4}).

The asymptotics \eqref{predely} is a particular case of Theorem \ref{Asymptotics} describing the asymptotic behaviour of $\Phi(r)$ under the general assumptions on the kernel $a(x)$. In our case $\Phi_G(r) = \Phi(\xi_r^{-1})$, and if we take $p=2, b= \frac14$, then $I(s) = \frac14 s^2$, and we immediately obtain \eqref{predely} from \eqref{Phi-inftyG1}.

\end{proof}

\subsection{The case $|x| \ge t^{\frac{1+\delta}{2}}, \; \delta>1$}

In this subsection we consider a region in $(x,t)$-space of super-large $|x|$, where $|x|>t^{(1+\delta)/2}$ with $\delta>1$.
In this case we again begin with the description of $\max\limits_z \mathcal{S}(z,t,x)$, where 
$$
\mathcal{S}(z,t,x) \ = \ \big(z \ln t - z \ln z + z - \frac{x^2}{4 z} - c(d) \ln z\big)\quad , \quad z > 0.
$$
Since
$t\to \infty$, we can omit the last term in (\ref{S}), and write as above the following equation on $\hat z = \hat z (t,x) = \mbox{argmax} \
S(z,t,x)$ :
\begin{equation}\label{dS}
\frac{\partial}{\partial z} \mathcal{S}(\hat z,t,x) \ = \ \ln t - \ln \hat z + \frac{x^2}{4 {\hat z}^2} \ = \ 0,
\end{equation}
or equivalently,
\begin{equation}\label{dS-1}
x^2 \ = \ 4 {\hat z}^2 ( \ln \hat z - \ln t) \ = \ 4 {\hat z}^2 \ln \frac{\hat z}{t}.
\end{equation}
Taking the logarithm on both sides  of equation (\ref{dS-1}) we obtain
$$
\ln \hat z \ = \  \ln |x| \ (1+o(1)),
$$
consequently equality (\ref{dS-1}) can be rewritten as
$$
x^2 = 4 {\hat z}^2 \ln \frac{|x|}{t} (1+ o(1)), \quad t \to \infty.
$$
Substituting $\frac{|x|}{2\sqrt{\ln \frac{|x|}{t}}}$ for $\hat z$ in (\ref{S}), we get
$$
S(\hat z, t, x) \ = \  \frac{|x|}{2 \sqrt{\ln \frac{|x|}{t}}} \ln t - \frac{|x|}{2 \sqrt{\ln \frac{|x|}{t}}} (\ln |x| - \ln 2 -
\frac12 \ln \ln \frac{|x|}{t}) + \frac{|x|}{2 \sqrt{\ln \frac{|x|}{t}}} - \frac{x^2 \sqrt{\ln \frac{|x|}{t}}}{2 |x|} + O(\ln |x|) \ =
$$
$$
 - |x| \sqrt{\ln \frac{|x|}{t}} \ (1+ o(1)),  \quad t \to \infty.
$$
Since $|x|>t^{(1+\delta)/2}$ with $\delta>1$, we can take $n_0 = |x|$ in (\ref{n})for
large enough $t$. Then as above we get the following two-sided estimate on $v(x,t)$:
\begin{equation*}\label{v4}
 e^{-t} \ e^{-  |x| \sqrt{\ln \frac{|x|}{t}} (1+ o(1)) } \le  v(x,t)  \le  |x|
 e^{-t}  e^{- |x| \sqrt{\ln \frac{|x|}{t}}  (1+ o(1))}.
\end{equation*}
Since $t=o(|x|)$, this yields
\begin{equation}\label{Rx}
\frac{\ln v(x,t)}{ |x| \sqrt{\ln \frac{|x|}{t}}} \ \to \ - 1, \quad \mbox{ as } \; t \to \infty, \quad |x|>t^{(1+\delta)/2}, \; \delta>1.
\end{equation}

\noindent
{{\bf Conclusions.}
\begin{enumerate}
\item If $|x| \le r t^{1/2}$, then the main term of the asymptotics of $v(x,t)$ coincides with the classical heat kernel $p_t(x,0)$ defined by \eqref{lhe}.
\item If $|x| \le r t^{\frac12+\frac{\delta}{2}}, \; 0<\delta<1$, then the main term of the logarithmic asymptotics of $v(x,t)$
coincides
with that of the classical heat kernel.
\item If $|x| = r t$, then the leading term of the logarithmic asymptotics of $v(x,t)$ is a linear function $\Phi_G(r)t$. , The leading term of the logarithmic asymptotics of the classical heat kernel is also a linear function $(r^2/4)t$. However, the corresponding coefficient $\Phi_G(r)$ is strictly less than $r^2/4$ for all $r>0$. This reflects the fact that in this range of $x$ the non-local heat kernel has more heavy tail than the classical one. It should also be noted that the coefficient
    $\Phi_G(r)$ is close to $r^2/4$ for small $r$ while $\Phi_G(r)\ll r^2$ for large $r$.
%
%
\item If $|x| \ge r t^{1+\delta}, \; \delta>1$, then the main term of the logarithmic asymptotics of $v(x,t)$ given by \eqref{Rx}
differs
essentially from the logarithmic asymptotics of the classical heat kernel, in particular $v(x,t)$ has more heavy tail, than the classical
heat
kernel.
\end{enumerate}
}

\section{Kernels with generic light tails}

\subsection{Main results}

\setcounter{equation}{0}
In this section we consider generic non-local operators with  convolution kernels that have  light
tails at infinity.
More precisely, we
assume that, in addition to \eqref{a1}--\eqref{a2},
the convolution kernel $a(x)$ satisfies for some  $p \ge 1$ the following condition
\begin{equation}\label{lt}
0 \le a(x) \le C_1 e^{- b |x|^p} \; ,
\end{equation}
or even a more strong condition
\begin{equation}\label{boundedsupp}
a(x) \in C_0(R^d), \quad \; \mathrm{supp} \ a(x) \subset K_\mu = \{ x \in R^d: \ |x| \le \mu \} \qquad \mbox{for some } \; \mu>0.
\end{equation}
In what follows we assume that $\mu$ is chosen in the optimal way, that is $\mu = \min \{ \tilde \mu>0: \ supp \, a \subset K_{\tilde \mu} \}$.

Since, in contrast with the Gaussian case, here $a^{\ast k}$ do not admit an explicit formula for $k\geq 1$,  we have to obtain sharp enough estimates for these higher order convolutions.  To this end we first make use of the results on the asymptotic behaviour of distributions of  the sums of i.i.d. random variables,  such as the local  central limit theorem and the large deviations principle, and then combine these results  with analytic techniques in order to obtain the asymptotics for $v(x,t)$.

\medskip
As in the previous section, for large $t$ four different regions of $x$ are considered: \\[2mm]
1) $|x| \le r t^{1/2}(1+ o(1))$ (standard deviations region) \\
2) $|x| = r\, t^{\frac{1+\delta}{2}} (1+ o(1)), \ 0<\delta<1$ (moderate deviations region) \\
3) $|x| = r t (1+ o(1))$ ($\delta=1$) (large deviations region) \\
4) $|x| = rt^{\frac{1+\delta}{2}} (1+ o(1)), \ \delta > 1$ ("extra-large" deviations region)

\medskip\noindent
The next two theorems describe the asymptotic behaviour of  $v(x,t)$  in the regions 1,\,2, and 4.

\begin{theorem}[The regions of standard and moderate  deviations]\label{LTail-1var2}
 Assume that $a(x)$  satisfies \eqref{a1}--\eqref{a2} and \eqref{lt}. Then for the function $v(x,t)$
 the following asymptotic relations hold as $t \to \infty$: \\
1) if $|x| \le r t^{\frac12}$ for some $r>0$, then
\begin{equation}\label{LT1}
v(x,t) = \frac{c(\sigma)}{ t^{\frac{d}{2}}} \ e^{-\frac{(\sigma^{-1}x,x)}{2 t}} \left( 1 +  o (1) \right),
\end{equation}
where $c(\sigma)=(2\pi)^{-\frac d2}|\mathrm{det(\sigma)}|^{-\frac12}$, $\sigma$ is the covariance matrix of the distribution $a(x)$;\\[1mm]
2) if $x = r t^{\frac{1+\delta}{2}} (1+ o(1))$ with  $0<\delta<1$ and $r \in R^d \backslash \{0 \}$, then
\begin{equation}\label{Tless1}
v(x,t) = e^{- \frac{(\sigma^{-1}x,x)}{2 t}(1 + o(1))} =  e^{-\frac12 (\sigma^{-1}r,r) \, t^\delta (1+ o(1))}.
\end{equation}
\end{theorem}

\begin{theorem}[The regions of extra-large deviations]\label{LTail-1var3}
 Assume that $a(x)$  satisfies \eqref{a1}--\eqref{a2} and \eqref{lt}. Then  for $|x| = r t^{\frac{1+\delta}{2}}(1+o(1))$ with  $ \delta>1$ and $r>0$ the following asymptotic upper bound holds:
\begin{equation}\label{Pge1}
v(x,t) \le e^{- c_p t^{\frac{\delta+1}{2}} (\ln t)^{\frac{p-1}{p}} (1+o(1))  },\qquad \hbox{as } t \to \infty,
\end{equation}
where the constant $c_p= c_p (b,r)$ depends on $b$, $r$ and $p$.
\\[3mm]
  If in addition $a(x)$ satisfies \eqref{boundedsupp}, then for $|x| = r t^{\frac{\delta+1}{2}}  (1+o(1)) $ with $\delta>1$
\begin{equation}\label{Pfinite}
v(x,t) \le e^{- \tilde c(\mu) \: t^{\frac{\delta+1}{2}} \ln t(1+ o(1))},\qquad \hbox{as } t \to \infty,
\end{equation}
where $\tilde c(\mu) = \frac{(\delta-1)r}{2\mu}$.
\end{theorem}

\medskip

In the region $x \sim t$, usually called  large deviations region,  our approach relies essentially on the properties of the rate function $I(r)$ of the sum of i.i.d. random variables.
From now on $S_k$ stands for the sum of  i.i.d. random variables (vectors) $X_1, \ldots, X_k$ with common distribution $a(x)$.
From \eqref{lt} 
it follows that the random variables $X_j$ have exponential moment $\Lambda(\gamma)=\mathbb{E} e^{\gamma X_1}$ for all $\gamma$ from a neighborhood of $0$ (the so-called Cramer condition). Under this condition the large deviation principle holds for $S_k$ with a rate function
\begin{equation}\label{I(r)}
I(r) = \sup\limits_{\gamma} \, (\gamma \cdot r - L(\gamma)), \quad r, \gamma \in R^d,
\end{equation}
where $I(r)$ is the Legendre transform of the cumulant generating  function $ L(\gamma) = \ln \Lambda(\gamma)$, and $\gamma \cdot r$ stands for the scalar product in $R^d$.

In order to formulate the main result of this section we denote by $\xi_r$  a positive solution
of the equation
\begin{equation}\label{eqeq}
\ln \xi = I(\xi r) - \xi r \cdot \nabla I(\xi r), \quad \xi \in R,
\end{equation}
and introduce the function
\begin{equation}\label{Phi(r)}
\Phi(r) = 1 -\frac 1{\xi_r} \big( 1+\ln \xi_r - I( \xi_r r) \big).
\end{equation}
Equation \eqref{eqeq} has a unique solution $\xi_r>0$ for
any $r \in R^d\setminus \{0\}$, moreover  $0< \xi_r <1$, see Lemma \ref{Xi} below.

We introduce now additional technical conditions on the kernel.\\
$({\bf A}_{1})$
in the case $p=1$ for any $b_1 > b$ and any $\theta \in S^{d-1}$
\begin{equation}\label{two-sided-bis}
\mathbb{E} e^{b_1 X  \cdot \theta} = \infty,
\end{equation}
where $b$ is the same constant as in \eqref{lt}. \\
$({\bf A^s_1})$
in the case $p=1$ for any $\theta \in S^{d-1}$
\begin{equation*}\label{stwo}
\mathbb{E} |X| e^{b  X \cdot \theta} = \infty.
\end{equation*}
$({\bf A}_{p})$
in the case $p>1$
\begin{equation}\label{two-sided_p-bis}
L(\gamma) =\ln \mathbb{E} e^{\gamma  \cdot X} = C(b,p)|\gamma|^{p/(p-1)}(1+o(1)),\quad\hbox{as } |\gamma|\to\infty,
\end{equation}
where  $C(b,p)=\frac{p-1}p (bp)^{-1/(p-1)}$ is a constant appearing in the logarithmic asymptotics of the Laplace transform of $e^{-b|x|^p}$.

\begin{remark}{\rm
Condition  ${\bf A}_{p}, \ p \ge 1,$ can be treated as a sort of soft lower bound for $a(x)$. In particular, it holds if
$a(x)$ satisfies the following two-sided estimate
\begin{equation*}\label{2side}
C_2 e^{-b |x|^p} \le a(x) \le C_1 e^{-b |x|^p}, \quad p\geq 1.
\end{equation*}
Observe also that under condition  ${\bf A}_{p}, p \ge 1,$ the function $a(x)$ can not satisfy \eqref{boundedsupp}.

It should be  emphasized that in the case $p=1$ conditions ${\bf A_1}, {\bf A^s_1}$ are required for proving the main result on the asymptotics of the heat kernel, while in the case $p>1$ condition ${\bf A}_{p} $ is only used for determining the asymptotic behaviour of the function $\Phi(r)$ for large $r$.}
\end{remark}

\begin{theorem}[Asymptotic upper bounds]\label{LTail-2_MD}
Let conditions \eqref{a1}--\eqref{a2} and \eqref{lt} be fulfilled , and assume additionally that in the case
$p=1$ condition ${\bf A_1}$ holds.
Then for any $r \in R^d\backslash \{ 0 \}$ and for $x= rt (1+o(1))$  the following asymptotic estimate holds as $t \to \infty$:
\begin{equation}\label{AsympMD}
v(x,t) \le e^{- \Phi(r) t (1+ o(1))},
\end{equation}
where the function $\Phi(r)$ is defined by \eqref{Phi(r)}.


Moreover, $\Phi(0)=0$,  $\Phi(r)>0$, if $r \neq 0$, $\Phi$ is a convex function, and the following limit relations hold:
\begin{eqnarray}\label{Phi0-bis2}
&\Phi(r) = \frac12 \, \sigma^{-1}r \cdot r \, (1+ o(1)), \quad& \mbox{ as } \; r \to 0;\\[2mm]
\label{Phiinfty-bis2}
&\Phi(r) \to \infty, \quad&\mbox{ as } \; r \to \infty.
\end{eqnarray}
If $p=1$, then
\begin{equation}\label{Phi-infty1-bis2}
 \Phi (r) = b|r| \, (1+o(1)),  \quad \mbox{ as } \; |r| \to \infty.
 \end{equation}
 If $p>1$ and condition  ${\bf A}_{p}$ holds, then
 \begin{equation}\label{Phi-inftyG1-bis2}
 \Phi (r) =\frac p{p-1} \big(b(p-1)\big)^{1/p} |r| (\ln |r|)^{\frac{p-1}{p}} \, (1+o(1)), \quad \mbox{ as } \; |r| \to \infty.
 \end{equation}
 If  condition \eqref{boundedsupp} holds, then
\begin{equation}\label{Phibs-bis}
 \Phi (r) \ge  \frac{1}{\mu} \, |r|\ln | r| \quad \mbox{ as } \; |r|\to \infty.
\end{equation}
\end{theorem}

\begin{remark}{\rm
Notice that under the assumptions of Theorem  \ref{LTail-2_MD} the  function $\Phi(r)$ need not be isotropic. This is illustrated by formula \eqref{Phi0-bis2}. However,  the additional condition ${\bf A}_{p}$ ensures that for large $|r|$ the principal term of the asymptotics of $\Phi(r)$ is radially symmetric, see \eqref{Phi-infty1-bis2}--\eqref{Phi-inftyG1-bis2}.}
\end{remark}

\begin{corollary}[Spherically symmetric kernels]
Let $a(x)=a(|x|), x \in R^d,$ be a spherically symmetric kernel satisfying all the conditions of Theorem \ref{LTail-2_MD}.
Then for any $s>0$ and for $|x| = s t (1+o(1))$ the following asymptotic estimate holds as $t \to \infty$:
\begin{equation}\label{LT2}
v(x,t) \le e^{ - \Phi(s) t (1+ o(1))},
\end{equation}
where the function $\Phi(s)$ is defined in \eqref{Phi(r)},
and formulae \eqref{Phi0-bis2} - \eqref{Phibs-bis} from Theorem \ref{LTail-2_MD} take an easier form, namely
\begin{eqnarray}\label{Phi0}
\Phi(s) = \frac{s^2}{2 \, \sigma}(1+ o(1)), \quad &&\mbox{ as } \; s \to 0;\\
\label{Phi-inftyG}
 \Phi (s) \to \infty, \quad &&\mbox{ as } \; s \to \infty
\end{eqnarray}
If $p=1$, then
\begin{equation}\label{Phi-infty1}
 \Phi (s) = bs \, (1+o(1)), \quad s \to \infty.
 \end{equation}
If $p>1$ and in addition conditions ${\bf A}_{p}$ hold, then
 \begin{equation}\label{Phi-inftyG1}
 \Phi (s) =\frac p{p-1} \big(b(p-1)\big)^{1/p} s (\ln s)^{\frac{p-1}{p}} \, (1+o(1)), \quad s \to \infty .
 \end{equation}
 If  \eqref{boundedsupp} holds, then
\begin{equation}\label{Phibs}
 \Phi (s) \ge  \frac{1}{\mu} \, s\ln  s \quad \mbox{ as } \; s\to \infty.
\end{equation}
\end{corollary}

\begin{remark}{\rm  Observe that relation \eqref{LT2}  and the asymptotics in \eqref{Phi0} and \eqref{Phibs} of $\Phi$ coincide with the
estimates of the heat kernel on graphs \cite{Davies} (see \eqref{Dav}--\eqref{davies}).}
\end{remark}

\medskip

Using another approach that relies on some exponential transformation of the random variable with density $a(x)$ under slightly more strong condition (for $p=1$) we can show that the upper bound obtained in Theorem \ref{LTail-2_MD} gives in fact the large time asymptotics of the fundamental solution. The following statement holds.

\begin{theorem}[Large time asymptotics]\label{Asymptotics}
Let conditions \eqref{a1}--\eqref{a2} and \eqref{lt} be fulfilled , and assume additionally that in the case
$p=1$ condition ${\bf A^s_1}$ holds.
Then for any $r \in R^d\backslash \{ 0 \}$ and  $x= rt (1+o(1))$
\begin{equation}\label{AsympMD-bis}
v(x,t) = e^{- \Phi(r) t (1+ o(1))} \quad \mbox{as } \; t \to \infty,
\end{equation}
where the function $\Phi(r)$ is defined by \eqref{Phi(r)} and possesses all the properties enumerated in Theorem \ref{LTail-2_MD}.
\end{theorem}


\subsection{Properties of $I(r)$ and $\Phi(r)$}

We preface the proof of the theorems by a number of technical statements.
We discuss in this section the asymptotic properties of the function $I(r)$ defined by \eqref{I(r)} that will be used further in the analysis of the function $v(x,t)$ in the regions of moderate and large deviations. Due to the symmetry of $a(x)$ stated in \eqref{a1} the functions $I(r)$ and $L(\gamma)$ are symmetric with respect to zero, that is $I(-r)=I(r)$ and $L(-\gamma) = L(\gamma)$.
We denote by $\cal A$ the convex hull of the support of $a(\cdot)$. From our conditions \eqref{a1}--\eqref{a2} it follows that
$\cal A$ contains a neighbourhood of zero. Notice that the set $\cal A$ is symmetric with respect to the origin.

First we consider the 1-D case. In this case, ${\cal A} = [\inf {\rm supp} \, a, \, \sup {\rm supp} \, a] = [- \mu, \mu]$.

\begin{proposition}[1-D case]\label{I}
1. For any distribution $a(x)$ satisfying  \eqref{a1}--\eqref{a2} and\eqref{lt}  we have
\begin{equation}\label{I0}
I(s) = \frac{s^2}{2\sigma}(1+o(1)) \quad \mbox{ as } \; s \to 0.
\end{equation}
2. If the distribution $a(x)$ in addition satisfies condition  ${\bf A}_p$, $p\geq 1$,
then $I(s)$ has the following asymptotics as $s \to \infty$:
\begin{equation}\label{Iinfty}
 I(s) = b s (1+o(1)), \quad \mbox{if} \quad p=1; \qquad I(s) =  b s^p (1+o(1)), \quad \mbox{if} \quad p>1,
\end{equation}
where $b$ is the same constant as in  \eqref{lt}. If $p>1$ then
\begin{equation}\label{gti}
\lim\limits_{s\to\infty}\frac{I(s)}{|s|}=+\infty.
\end{equation}
3. If the distribution $a(x)$ in addition  satisfies \eqref{boundedsupp}, then $I(s)$ is a smooth function on $(- \mu, \mu)$,
$$
I(s) \to \infty \quad \mbox{ as } \; s \to \mu-0 \; \mbox{ or } \; s \to -\mu+0,
$$
and  $I(s) = \infty$ if $|s| \ge \mu$.
\end{proposition}
\begin{proof}
1. By the definition of $I(s)$ considering the smoothness of $L(\gamma)$ in the vicinity of zero we have
$I(s) = s\gamma^{\star} - L(\gamma^\star)$,
where $\gamma^\star = \gamma^\star(s)$ is the solution of equation $s=L'(\gamma)$.  Using the Taylor decomposition for $L'(\gamma)$ about zero by the implicit function theorem we obtain $\gamma^\star = \frac{s}{L''(0)} (1+o(1)) = \frac{s}{\sigma} (1+o(1))$ for small enough $s$. Consequently, 
$$
I(s) = s\gamma^{\star} - L(\gamma^\star) = \frac{s^2 (1+o(1))}{\sigma} - \frac12 {\gamma^\star}^2 L''(0)   = \frac{s^2}{2 \, \sigma} (1+o(1)).
$$
\medskip

2. In the case $p=1$ conditions \eqref{lt} and \eqref{two-sided-bis} on the distribution $a(x)$ imply that $\Lambda(b_1) = \infty$ for any $b_1>b$, and $\Lambda(b')$ is finite for all $0<b'<b$. Therefore, for $s \ge 0$
\begin{equation}\label{propI}
I(s)< b s, \quad I'(s) \le b, \mbox{ and } \quad \lim\limits_{s \to \infty} I'(s) \le b.
\end{equation}
The last limit exists since $I'(s)$ is monotone and bounded. If we assume that $\lim\limits_{s \to \infty} I'(s) =a < b$, then taking $\beta=\frac{a+b}{2}$ we get
\begin{equation}\label{contr}
I(s) = \sup\limits_{\gamma<b}(\gamma s - L(\gamma)) >  \beta s - L(\beta) = \beta s (1+ o(1)) \quad\mbox{ as } \; s \to \infty.
\end{equation}
On the other hand, if  $\lim\limits_{s \to \infty} I'(s) =a< \beta$, then $I(s)< a s (1+ o(1))$ as $s \to \infty$, which  contradicts \eqref{contr}. Thus $\lim\limits_{s \to \infty} I'(s) = b$, and the first formula in \eqref{Iinfty} follows.

In the case $p>1$ due to \eqref{two-sided_p-bis}
the solution $\gamma^\star$ of equation $L'(\gamma) = s$ has the asymptotics $\gamma^\star = b p \, s^{p-1}(1+o(1)) $ as $s \to \infty $, and thus
$$
I(s) = \sup\limits_\gamma ( s \gamma - L(\gamma)) =  s \gamma^\star - L(\gamma^\star) = b s^p (1+o(1)).
$$
Limit relation \eqref{gti} follows from the fact that for $p>1$ the function $L(\gamma)$ is finite for all $\gamma\in\mathbb R$.
Then for any $N>0$
$$
I(x)=\sup\limits_\gamma(s\gamma-L(\gamma))\geq sN-L(N),
$$
and thus $\liminf\limits_{s\to\infty}\frac{I(s)}{s}\geq N$, which yields \eqref{gti}.

\medskip
3. Since $\mathrm{supp} \ a \subset [-\mu, \mu]$, then, for $\gamma \ge 0$, $\Lambda (\gamma) = \int e^{\gamma x} a(x) dx \le  e^{\gamma \mu}$. Consequently,
$L(\gamma) \le \gamma \mu$, and for $s> \mu$ we have
$$
I(s) = \sup\limits_\gamma (\gamma s - L(\gamma)) \ge \sup\limits_\gamma (s- \mu)\gamma = \infty.
$$
On the other hand, since $\mu = \sup {\rm supp}\, a$, for $\gamma \ge 0$ and for any $\delta>0$
$$
\Lambda (\gamma) = \int e^{\gamma x} a(x) dx \ge c_\delta  e^{\gamma (\mu - \delta)}
$$
for some $c_\delta>0$. Thus,
$$
L(\gamma) \ge \ln c_\delta + \gamma (\mu - \delta) \quad \mbox{ and } \quad I(s) \le \sup\limits_\gamma (s- (\mu-\delta))\gamma - \ln c_\delta < \infty,
$$
if $s< \mu-\delta$. Since we take an arbitrary $\delta>0$, then $I(s)$ is finite for all $s \in (- \mu, \mu)$. The smoothness of $I(s)$ follows from the standard convexity arguments.

It remains to prove that $I(s) \to \infty$ as $s \to \mu-0$. Since $a(x) \le C_1$, we have
$$
\Lambda (\gamma) \le C_1 \int\limits_{-\mu}^{\mu} e^{\gamma x}  dx = \frac{C_1}{\gamma} (e^{\mu \gamma} - e^{-\mu \gamma}) < \frac{C_1}{\gamma} e^{\mu \gamma}.
$$
Then
$$
L(\gamma) < - \ln \gamma  + \gamma \mu + \ln C_1,
$$
and
$$
I(s) \ge \sup\limits_\gamma \left( (s- \mu)\gamma + \ln \gamma \right) - \ln C_1 \ge  (s- \mu)\gamma^* (s) + \ln \gamma^* (s)  - \ln C_1,
$$
where $\gamma^* (s)= \frac{1}{\mu -s} $ is the argmax  of the function $(s- \mu)\gamma + \ln \gamma $. Since $\gamma^* (s)  \to \infty $, as $s \to \mu -0$, then
$$
I(s) \ge  \ln \gamma^* (s)  - \tilde C \to +\infty, \quad \hbox{as } s \to \mu -0.
$$
 The statement for negative $s$ follows from the symmetry of $a$.
\end{proof}

\medskip

Next we describe the properties of the rate function $I(r)$ in the multidimensional case.

\begin{proposition}[Multi-dimensional case]\label{I_MD}
1. For any distribution $a(x)$ satisfying  \eqref{a1}--\eqref{a2} and \eqref{lt} we have
\begin{equation}\label{I0-bis}
I(r) = \frac12 \, \sigma^{-1} r \cdot r \, (1+o(1)) \quad \mbox{ as } \; r \to 0.
\end{equation}
2. If $p=1$, and in addition to the above conditions  ${\bf A}_1$ is fulfilled,
then $I(r)$ has the following asymptotics:
\begin{equation}\label{Iinfty-bis}
  I(r) = b |r| (1+o(1)), \qquad  \nabla I(r) = b \frac{r}{|r|} (1+ o(1)), \quad\hbox{as } |r| \to \infty,
\end{equation}
where $b$ is the same constant as in  \eqref{lt}. Moreover, $|\nabla I(r)| \le b$ for all $r \in R^d$.

If $p>1$, and in addition to \eqref{a1}--\eqref{a2} and\eqref{lt} the function $a(x)$ satisfies condition ${\bf A_p}$,
then
\begin{equation}\label{Iinfty-bisbis}
I(r) =  b |r|^p (1+o(1)), \qquad \mbox{ as } \; |r| \to \infty.
\end{equation}
3. If \eqref{boundedsupp} holds, then $I(r)$ is a smooth function in the interior of the convex hull $\cal A$. Moreover,
$I(r) \to \infty$ as $ {\rm dist}(r, \partial {\cal A}) \to 0$, and $I(r) = \infty$ for all $r \in R^d \backslash {\cal A}$.

\end{proposition}

\begin{proof}
The proof of this proposition is mostly based on the same arguments as  the proof of Proposition \ref{I}.

1. Using the Taylor decomposition for $ L(\gamma)$ about zero we obtain as above $\gamma^\star = (\nabla \nabla L(0))^{-1} r (1+o(1))$ for small enough $r$. Consequently,
$$
I(r) = (\nabla \nabla L(0))^{-1}r \cdot r - \frac12 \nabla \nabla L(0) \gamma^\ast \cdot \gamma^\ast + o(r^2) =  \frac12 \, (\nabla \nabla L(0))^{-1}r \cdot r + o(r^2) = \frac12 \,\sigma^{-1}r \cdot r + o(r^2),
$$
since $\sigma = \nabla \nabla L(0)$, and the asymptotics \eqref{I0-bis} follows.

2. In the case $p=1$ conditions \eqref{lt} and \eqref{two-sided-bis} on the distribution $a(x)$ imply that for any $\theta \in S^{d-1}$
$$
\Lambda(b_1 \theta) = \mathbb{E} e^{b_1 \theta \cdot X} = \infty, \quad \mbox{ if } b_1>b \qquad \mbox{ and } \quad  \Lambda(b_1 \theta)< \infty  \quad \mbox{ if } b_1<b.
$$
Therefore,
\begin{equation}\label{I(r)P6}
I(r) = \sup_{\gamma \in R^d} (r \cdot \gamma - L(\gamma)) = \sup_{|\gamma| \le b} (r \cdot \gamma - L(\gamma)) <  |r| b.
\end{equation}
The function $I(s \theta)$ is a convex function of $s \in R^1$ for any $\theta \in S^{d-1}$. Consequently,  \eqref{I(r)P6} implies inequality
\begin{equation}\label{mitlioc}
|\nabla I(r)| \le b \qquad \forall r \in R^d.
\end{equation}
In the same way as in Proposition \ref{I} using the convexity of  $I(s \theta)$ we obtain 
$$
b|r| (1+ o(1)) \le I(r) \le b|r|, \qquad \mbox{ and } \quad \big( \nabla I(r) \cdot \frac{r}{|r|} \big) \to b \quad \mbox{ as } \; |r| \to \infty.
$$
Combining the last relation with \eqref{mitlioc} we obtain the second equality in  \eqref{Iinfty-bis}.

In the case $p>1$ considering the convexity of $L(\gamma)$ with the help of the implicit function theorem we get that the solution $\gamma^\ast \in R^d$ of equation $\nabla L(\gamma) = r$ has the asymptotics
$$
\gamma^\ast = b p |r|^{p-2} r (1+ o(1)) \quad \mbox{ as } \; r \to \infty.
$$
This implies \eqref{Iinfty-bisbis}.

3. Denote by $G(r)$ the following auxiliary function:
$$
G(r) \ = \left\{
\begin{array}{l}
0, \qquad r \in {\cal A}, \\
+ \infty, \quad r \not \in {\cal A}.
\end{array}
\right.
$$
Then the Legendre transform of $G$ is equal to $G^*(\gamma) = \mu ( \frac{\gamma}{|\gamma|} ) \, |\gamma|$, where
$$
\mu(\theta) = \sup_{r \in {\cal A}} r \cdot \theta = \sup_{r \in {\rm supp} \, a} r \cdot \theta, \quad \theta \in S^{d-1}.
$$
In the same way as in the proof of Proposition \ref{I} one can show that
\begin{equation}\label{b1}
L(\gamma) = \ln \mathbb{E} \, e^{\gamma \cdot X} = \mu ( \frac{\gamma}{|\gamma|} ) \, |\gamma| (1+ o(1)), \quad |\gamma| \to \infty,
\end{equation}
and moreover,
\begin{equation}\label{b2}
L(\gamma) \le \mu ( \frac{\gamma}{|\gamma|} ) \, |\gamma| - \ln |\gamma| +C
\end{equation}
for some constant $C$.

Since $G^{* *}(r) = G(r)$, comparing \eqref{b1} with $G^*(\gamma)$ we conclude that $I(r) = + \infty$ in $R^d \backslash {\cal A}$ and $I(r) < \infty$ for $r$  in the interior of $\cal A$. The fact that $I(r) \to \infty$ as $dist \, (r, \partial {\cal A}) \to 0$ can be justified  in the same way as in the proof of Proposition \ref{I} using inequality \eqref{b2}.
\end{proof}

\begin{lemma}\label{Xi}
Let $a(x)$ satisfy \eqref{a1}--\eqref{a2} and  \eqref{lt}. 
Then for any $r \in R^d \backslash \{0\}$ equation \eqref{eqeq} has a unique solution $\xi_r$ and $0< \xi_r < 1$.

\end{lemma}

\begin{proof}
If the convex hull $\cal A$ of supp $a$ coincides with $R^d$, then differentiating the right-hand side of \eqref{eqeq} in $\xi$ we obtain
$$
r \cdot \nabla I (\xi r ) - r \cdot \nabla I (\xi r ) - \xi r \cdot \nabla \nabla I (\xi r ) r = - \xi r \cdot \nabla \nabla I (\xi r ) r \le 0
$$
because of  convexity of $I$; here $\nabla \nabla$ denotes the Hessian. Moreover, for sufficiently small $\xi$ we have
 $-  r \cdot \nabla \nabla I (\xi r ) r < 0$.
Thus the function on the right-hand side of \eqref{eqeq} is  decreasing in $\xi$, and, since $I(0)=0$, we immediately conclude that  \eqref{eqeq} has a unique solution and  $0< \xi_r < 1$.

If ${\cal A} \neq R^d$, but the ray $\{ sr\}_{s \ge 0}$ lies inside $\cal A$, then we can use the same arguments as above. If the ray $\{ sr\}_{s \ge 0}$ intersects $\partial {\cal A} $ at a point $s^*_r$, then it follows from Proposition \ref{I_MD} that $I(sr) \to \infty$ as $s \to s^*_r -0$.
In addition, the convexity of $I(r)$ and the Newton-Leibniz formula imply
$$
\lim_{s \to s^*_r -0} \frac{\frac{d}{ds} I(sr)}{I(sr)} = \infty.
$$
Consequently,
$$
I(sr) - sr \cdot \nabla I(sr) = I(sr) - s \frac{d}{ds} I(sr) \to - \infty \quad \mbox{ as } \; s \to s^*_r -0,
$$
and again we obtain the unique solution $0< \xi_r <1$ of equation \eqref{eqeq}.
\end{proof}

\begin{proposition}
The function $\Phi(r)$ is a convex function, $\Phi(0) = 0$, and $\Phi (r)>0$ for any $r~\in~R^d~\backslash~\{0\}$. Moreover, if $a(x)$ satisfies \eqref{lt} with $p\geq1$ and, in the case $p=1$, also condition ${\bf A_1^s}$, then $\Phi$ is strictly convex: $\nabla \nabla \Phi(r) r \cdot r>0$.
\end{proposition}

\begin{proof}
If $r=0$, then \eqref{eqeq} implies that $\xi_0 = 1$, and $\Phi (0)=0$.
Let us show that $\nabla \Phi(r) \cdot r>0$ for any $r \in R^d \backslash \{0\}$.
Indeed,  $\Phi(r) = \Phi(\xi(r))$ with $\xi(r) = \xi_r$, then using  \eqref{eqeq} and considering the properties of  $I(r)$ we have
$$
\nabla \Phi(r) = \frac{\nabla \xi(r)}{ \xi^2(r)} \big[ \ln \xi(r) - I( \xi(r) r) +  \xi(r) r\cdot\nabla I(\xi(r) r)  \big] +  \nabla I(  \xi(r) r) = \nabla I(\xi(r) r).
$$
Consequently,
$\nabla \Phi(r)\cdot r = \nabla  I(\xi(r) r)\cdot  r >0$ and $\Phi (r)>0$ for any $r~\in~R^d~\backslash~\{0\}$.

To prove the convexity of $\Phi$ we differentiate equation \eqref{eqeq} in $r$ and obtain
$$
\nabla \xi (r)= - \xi^2(r) \nabla \nabla I(\xi(r) r)\,r\big[\xi(r)+ r\cdot \nabla \xi(r)\big].
$$
The assumption $\nabla \xi(r) \cdot r > 0$ leads to a contradiction. Therefore, $\nabla \xi(r)\cdot r \le 0$ and $\xi(r)+ r \cdot\nabla \xi (r) \ge 0$ for all $r \in R^d$. This yields  the inequality  $\nabla \nabla \Phi(r) r \cdot r\ge 0$.
Additionally,  $\nabla \nabla I(r) r \cdot r>0$ and  $\xi(r)+ r\cdot\nabla \xi (r) >0$ in the case $p>1$ or $p=1$ under condition ${\bf A_1^s}$. This yields a strict convexity of $\Phi$.
\end{proof}

\begin{proposition}[Skewed distribution]\label{DR}
1. Let distribution $a(x)$ satisfy \eqref{a1}--\eqref{a2} and \eqref{lt}. Assume also that in the case $p=1$ condition $\bf A^s_1$ is fulfilled and in the case $p>1$ the following condition holds:
\begin{equation}\label{N0}
\int\limits_{x \cdot \theta >N} a(x) dx >0 \quad \mbox{for any } \; N>0 \; \mbox{ and any } \; \theta \in S^{d-1}.
\end{equation}
Then for any $x^\ast \in R^d$ equation
\begin{equation}\label{xast}
\nabla L (\gamma) = x^\ast
\end{equation}
 has a unique solution $\gamma^\ast \in R^d$ and, furthermore,  the following relations hold
\begin{equation}\label{N1}
I(x^\ast) = x^\ast\cdot \gamma^\ast - L(\gamma^\ast)
\end{equation}
\begin{equation}\label{N2}
\frac{1}{\Lambda(\gamma^\ast)} \int x a(x) e^{\gamma^\ast\cdot x} dx = x^\ast.
\end{equation}
Moreover, denoting $a_{\gamma}(x) = \frac{a(x) e^{\gamma\cdot x}}{\Lambda(\gamma)}$ we get
\begin{equation}\label{N3}
a^{\ast k}(k x^\ast) = a^{\ast k}_{\gamma^\ast}(k x^\ast) e^{-I(x^\ast) k}.
\end{equation}

2. If distribution $a(x)$ satisfies \eqref{lt} (or \eqref{boundedsupp}), then for any small enough $x^\ast \in R^d$  equation \eqref{xast} has a unique solution $\gamma^\ast$, and relations \eqref{N1}--\eqref{N3} hold.

\end{proposition}

\begin{proof}
1. Assume first that  $p>1$. From the properties of the function $L$ it follows that $\nabla L: R^d \to R^d$ is a semicontinuous strictly monotone operator, i.e. $(\nabla L(\gamma_1)~-~\nabla L(\gamma_2))\cdot(\gamma_1~-~ \gamma_2) > 0 $. Condition \eqref{N0} implies that
$\Lambda(\gamma) > e^{N \| \gamma \|}  c(N) $ with $c(N)>0$ for any $N$, and consequently
\begin{equation}\label{frac12N}
L(\gamma) > \frac12 N \|\gamma\| \quad \mbox{for all large enough } \; \| \gamma\|.
\end{equation}
 Since $L$ is a convex function, then $\frac{\nabla L(\gamma) \cdot\gamma}{\|\gamma \|}$ is monotonically increasing. This together with \eqref{frac12N} and $L(0)=0$ imply that
$$
\lim_{\| \gamma \| \to \infty} \frac{\nabla L(\gamma)\cdot \gamma}{\|\gamma \|} = +\infty.
$$
Then the unique solution of \eqref{xast} exists by the solvability theorem for monotone operators, see e.g. \cite{lions}.
If  with $p=1$, then under condition $\bf A^s_1$ using the Lebesgue theorem we obtain that
$
\lim_{\| \gamma \| \to b-0} \frac{\nabla L(\gamma)\cdot \gamma}{\|\gamma \|} = +\infty.
$
Then we can repeat the similar arguments to prove the existence of the unique solution of \eqref{xast}.

Equallity \eqref{N1} follows from the definition \eqref{I(r)} of the function $I(r)$.
Equality \eqref{N2} follows from \eqref{xast}. Equality \eqref{N3} is a direct consequence of relation \eqref{N1}.

2. If  $x^\ast \in R^d$ is small enough, then for any distribution $a(x)$ satisfying \eqref{lt}, \eqref{boundedsupp} equation  \eqref{xast} can be solved using the implicit function theorem.
\end{proof}

\subsection{The regions of standard and moderate  deviations. Proof of Theorem \ref{LTail-1var2}}

Under our standing assumptions the local central limit theorem applies to the sum of independent random variables with a common distribution $a(x)$, see for instance  \cite[Theorem 19.1]{BhaRao}. This implies the desired asymptotics \eqref{LT1} of $v(x,t)$ in the region  $|x| \le r t^{1/2}$ with an arbitrary $r>0$.



In this subsection we show that, in the region $x = r t^{\frac{1+\delta}{2}} (1+ o(1))$ with  $0<\delta<1$ and
$r\in \mathbb R^d\setminus \{0\}$, the asymptotics \eqref{Tless1} for $v(x,t)$ holds, as $t \to \infty$.
First we obtain the asymptotics for the $k$-th convolution power $a^{\ast k}(x)$ for large enough $k$.
Our approach essentially  relies on probabilistic arguments.


\begin{lemma}\label{deltaless1}
Let  conditions \eqref{a1}--\eqref{a2} be satisfied, and assume that \eqref{lt}  holds. Then
\begin{equation}\label{nconvolution1}
a^{\ast k} (x) =   e^{- \frac12\frac{\sigma^{-1}x \cdot x}{ k}(1 + o(1))}, \quad \mbox{ as } \;
 \; k \to \infty, \; \frac{|x|^2}k \to \infty, \;  \mbox{ and } \; \frac{|x|}{k}\to 0,
\end{equation}
where $\sigma $ is  the covariance matrix of the distribution $a(x)$, and $o(1)\to0$ as   $\frac{|x|^2}k \to \infty$.
\end{lemma}

\begin{proof}
Let $x^\ast = \frac{x}{k}$. Then $x^\ast \to 0$ as $|x| \to \infty, \ k \to \infty$, and using Proposition \ref{DR} we conclude that
the equation $\nabla L(\gamma) = x^\ast$ has a unique solution $\gamma^\ast = \gamma^\ast (x,k)$, where  $\gamma^\ast \to 0$ as $x^\ast \to 0$. Relation \eqref{N3} implies
\begin{equation}\label{akgamma1}
 a^{\ast k}(x) =  a^{\ast k}(k x^\ast) = a^{\ast k}_{\gamma^\ast}(k x^\ast) e^{-I(x^\ast) k} = a^{\ast k}_{\gamma^\ast}(k x^\ast) e^{-I(\frac{x}{k}) k}.
\end{equation}
It follows from \eqref{N2} and the local limit theorem for the sum of i.i.d. random variables with the common distribution $a_{\gamma^\ast}$ that
\begin{equation}\label{akgamma2}
a^{\ast k}_{\gamma^\ast} (k x^\ast) = \frac{\rm{det} \, \sigma_{\gamma^\ast}^{-1}}{(2 \pi k)^{d/2}}(1+o(1)) = \frac{\rm{det} \, \sigma^{-1}}{(2 \pi k)^{d/2}} (1+ o(1)), \quad \mbox{as } \; x^\ast \to 0 \; \; (\mbox{and} \ \gamma^\ast \to 0  ),
\end{equation}
where $\sigma$ is the covariance matrix for $a(x)$.

Finally using the asymptotic relation \eqref{I0-bis} for $I(\frac{x}{k})$ we obtain \eqref{nconvolution1} from \eqref{akgamma1} -- \eqref{akgamma2}.
\end{proof}

\begin{corollary}
If $x = r t^{\frac{1 + \delta}{2}}  (1+o(1))$ with  $ 0<\delta<1$ and $k \sim t$, then
$\frac{\sigma^{-1}x  \cdot x}k \sim (\sigma^{-1}r \cdot r) t^\delta$ and
\begin{equation}\label{nconvolution2}
a^{\ast k} (x) =  e^{-\frac12 ( \sigma^{-1}r \cdot r) \frac{ t^{1+\delta}}k (1 + o(1))}, \quad \mbox{ as } \;  k \to \infty.
\end{equation}
\end{corollary}

\medskip
Let us study now the asymptotic behaviour of the function $\frac{t^k a^{\ast k}(x)}{k!} $ as $t\to\infty$ and $x = r t^{\frac{1+\delta}{2}} (1+o(1)), \
0<\delta<1$. By Lemma \ref{deltaless1} and estimate (\ref{nconvolution2}) for any  constants $\alpha_1$ and $\alpha_2$ such that $0<\alpha_1<1<\alpha_2< \infty$, and for all $k$ from the interval $\alpha_1 t < k < \alpha_2 t$ the following asymptotics holds:
$$
\frac{t^k a^{\ast k}(x)}{k!} = \exp \left\{ k\ln t - k \ln k + k - \tilde c \: \frac{t^{1+\delta}}{k}(1+o(1)) \right\} = \exp S(k,t),  \quad t \to \infty,
$$
where $\tilde c = \frac12 \, \sigma^{-1}r \cdot r$ and $S(z,t)=  z\ln t - z \ln z + z - \tilde c \: \frac{t^{1+\delta}}{z}(1+o(1)) $.
If $\alpha_1$ is sufficiently small and  $\alpha_2$ is sufficiently large then the $\max\limits_{\alpha_1 t < z < \alpha_2 t } S(z,t)$
is attained in an interior point of the interval $(\alpha_1 t,\alpha_2 t)$, and the corresponding necessary condition reads
$$
\ln \frac{t}{z} + \tilde c \: \frac{t^{1+\delta}}{z^2} = 0.
$$
Setting $z = \xi t$ we arrive at the following equation for  $\xi$:
\begin{equation*}\label{xi.2}
\xi^2 \ln \xi = \tilde c \, t^{\delta-1}.
\end{equation*}
Since $ 0<\delta<1$, the right-hand side in this equation vanishes as $t\to\infty$. Therefore,
the solution $\hat \xi$ of this  equation
admits the representation $\hat \xi = 1+ \tilde c t^{\delta-1}(1+ o(1))$.
Consequently,
\begin{equation*}\label{hatz.2}
\hat z = \hat \xi t = t + \tilde c t^\delta + o(t^\delta), \qquad \max\limits_{\alpha_1 t < z < \alpha_2 t } S(z,t) =  S(\hat z,t) = t - \tilde c t^\delta + o(t^\delta),
\end{equation*}
and for any $\alpha_1$ and $\alpha_2$ such that $0<\alpha_1< 1 <\alpha_2< \infty$ we have
\begin{equation}\label{maxSz.2}
\max_{\alpha_1 t< k <\alpha_2 t } \frac{t^k a^{\ast k}(x)}{k!} \le e^{S(\hat z,t)} = e^{ t - \frac12 (\sigma^{-1}r \cdot r) t^\delta  + o(t^\delta) },  \quad
\mbox{as } t \to \infty.
\end{equation}

To  estimate $v(x,t)$ in the region $ x = r t^{\frac{1+\delta}{2}}(1+o(1))$ we split the sum in (\ref{v}) into
three parts:
\begin{equation}\label{v.2}
v(x,t) \ =  \ e^{-t}  \sum_{k=1}^{[t/2]} \frac{ t^k  a^{\ast k} (x)}{k!} +  e^{-t}  \sum_{k=[t/2]+1}^{[2t]} \frac{ t^k  a^{\ast k}
(x)}{k!} +  e^{-t}  \sum_{k > 2t} \frac{ t^k  a^{\ast k} (x)}{k!}.
\end{equation}
Considering the inequalities $k!>k^k e^{-k}$ and $a^{\ast k} (x) \le C_1$,
one can  estimate the first sun in \eqref{v.2} as follows
\begin{equation}\label{Ev.1}
e^{-t} \, \sum_{k=1}^{[t/2]} \frac{ t^k  a^{\ast k} (x)}{k!} \le  C_1 \frac{t}{2} \, e^{-t}  (2e)^{t/2} \le C_2 e^{-\beta t}, \qquad t
\to \infty
\end{equation}
with an arbitrary $\beta\in (0, \frac{1-\ln2}2)$.
For the third sum using the relation $\frac{t}{k}< \frac12$ we get
\begin{equation}\label{Ev.3}
e^{-t} \, \sum_{k>2t} \frac{ t^k  a^{\ast k} (x)}{k!} \le  C_1   e^{t - 2t \ln 2}.
\end{equation}
To estimate the second sum in \eqref{v.2} we use \eqref{maxSz.2}. This yields
\begin{equation}\label{Ev.2}
e^{-t} \; \sum_{k=[t/2]+1}^{[2t]} \frac{ t^k  a^{\ast k} (x)}{k!} \le  2t e^{-t} e^{S(\hat z,t)} \le e^{- \frac12 (\sigma^{-1}r \cdot r)\, t^\delta  +
o(t^\delta)}, \qquad t \to \infty.
\end{equation}
Finally, from \eqref{v.2} -- \eqref{Ev.2} we get the asymptotical upper bound in the region $x =r t^{\frac{1+\delta}{2}} (1+ o(1))$.

Taking in the sum \eqref{v.2} just one term that corresponds to $\max S(z,t)$ we obtain the lower bound
$$
v(x,t) \ge  e^{-t} e^{S(\hat z,t)} \ge e^{- \frac12 (\sigma^{-1}r \cdot r) \, t^\delta (1 +o(1))}.
$$
This completes  the proof of  \eqref{Tless1}.

\subsection{The regions of extra-large deviations: Proof of Theorem \ref{LTail-1var3}}

This section deals with the large time behaviour of $v(x,t)$ in the region $x\gg t$ which is associated with the "extra-large" deviations of the corresponding process. In this region we use the Markov inequality for estimating $\Pr (|S_k| > |x|)$.

\begin{lemma}\label{expp}
Let $X_i$ be i.i.d. 1-D random variables with a common distribution $a(x)$, satisfying
\eqref{a1}--\eqref{a2} and \eqref{lt}. Then there exist constants $\alpha_p = \alpha_p(b, p)$ and $\varkappa_p = \varkappa_p(b,p)$, such that for all $1\le k \le \alpha_p  x$ the following estimate holds
\begin{equation}\label{theta11}
P \{ S_k  > x \} \le
e^{- \varkappa_p \left(\frac{x}{k}\right)^p  k}.
\end{equation}
\end{lemma}

\begin{proof}
The cases $p=1$ and $p > 1$ are considered in a slightly different way. If $p=1$, the inequality
$$
{\mathbb E} e^{m X_1} \ \le \ e^{h m^2}
$$
being valid for all  $m\in(0,\frac b2)$ with some constant  $h >0$. Then the Markov inequality implies that
\begin{equation}\label{P2}
P \{ S_k  > x \} \le \ \min\limits_{0< m \le \frac b2 } \frac{(\mathbb{ E} e^{m  X_1})^k }{e^{m x}} = e^{ \min\limits_{0< m \le \frac{b}{2}} ( h m^2 k - m x) } = e^{ \frac{h b^2 k}{4} - \frac{b x}{2}} = e^{(\frac{h b^2 k}{4} - \frac{b x}{4}) - \frac{b x}{4}} \le  e^{- \frac{b x}{4}}
\end{equation}
for $k < \frac{x}{ h b}$. Thus in the case $p=1$ inequality \eqref{theta11} holds with $\varkappa_1 = \frac b4$ and $\alpha_1 = \frac{1}{2 h b}$.
\medskip

If $p>1$ then applying the Markov inequality we get
\begin{equation}\label{1}
P \{  S_k  > x \} \le \ \min\limits_{ m > 0 } \frac{(\mathbb{ E} e^{m X_1})^k }{e^{m x}} =
\left(  \min\limits_{m > 0}   \frac{\mathbb{ E} e^{m X_1} }{e^{\frac{m x}{k} }} \right)^k.
\end{equation}
Let us estimate $\mathbb{ E} e^{m X_1}$. Setting $\varphi(x) = mx - b x^p, \; x>0$, we
obtain
\begin{equation}\label{maxphi}
\max_x \varphi(x) = \varphi \left(\left( \frac{m}{ bp} \right)^{\frac{1}{p-1}}\right) = c_2 (b, p) \: m^{\frac{p}{p-1}}, \quad c_2 (b,p) = \frac{p-1}{p} (bp)^{-\frac{1}{p-1}},
\end{equation}
and $\varphi(x)<0$ as $x> x_1 = \left( \frac mb \right)^{\frac{1}{p-1}}$.
Since $mx-bx^p\leq \varphi'(x_1)(x-x_1)=m(1-p)(x-x_1)$ for $x\geq x_1$, from \eqref{maxphi} it follows that
$$
\mathbb{ E} e^{m  X_1} \le \int\limits_{|x|\le x_1} e^{\varphi(x)} dx + \int\limits_{|x|> x_1} e^{mx} a(x) dx \le
c_3 (b,p) m^{\frac{1}{p-1}} e^{c_2(b, p) \: m^{\frac{p}{p-1}}} + C_1((p-1)m)^{-1}.
$$
Then there exists a constant $c_4=c_4(p,b)$ such that for all $m\geq 1$
\begin{equation}\label{2}
\mathbb{ E} e^{m  X_1} \le e^{c_4 (b, p) \ m^{\frac{p}{p-1}}}.
\end{equation}
Inserting \eqref{2} into \eqref{1} yields
\begin{equation}\label{3}
P \{ S_k  >  x \}
\le  \left(  \min\limits_{m \ge 1}   e^{ c_4 (b, p) \: m^{\frac{p}{p-1}} - m \frac{x}{k}} \right)^k =
 e^{\ k \: \min\limits_{m\ge 1}   f_{p,k}(m)},
\end{equation}
where
$$
f_{p,k}(m) = c_4 (b, p) \: m^{\frac{p}{p-1}} - m \: \frac{x}{k}, \quad m\ge 1.
$$
If we take $k \le \alpha_p x$ with $\alpha_p \le \frac{p-1}{2p c_4}$, then $f_{p,k}(1)<0$ and $f'_{p,k}(1)<0$.
Determining the minimum of  $f_{p,k}$ we obtain
$$
\min\limits_{m \ge 1}   f_{p,k}(m) = - \varkappa_p \left(\frac{x}{k} \right)^p
$$
with some  constant $\varkappa_p>0$. Inequality \eqref{theta11} then follows from \eqref{3}.
\end{proof}

\begin{corollary}
In the multidimensional case estimate \eqref{theta11} takes the form
\begin{equation}\label{theta11-d}
P \{ |S_k|  > |x| \} \le
e^{- \varkappa_p \left(\frac{|x|}{k}\right)^p  k(1+o(1))},\qquad\hbox{as }|x|\to\infty.
\end{equation}
\end{corollary}

\begin{proof}
Given a sequence of i.i.d. random vectors $X_j$, $j=1,2,\ldots$, with a common distribution density $a(\cdot)$, for
any $\theta\in S^{d-1}$ we consider 1-D random variables $\theta \cdot X_j$. Denote the distribution density of $\theta\cdot X_j$
by $a_\theta(s)$. Then
$$
a_\theta(s)\leq C(1+s)^{d-1}e^{-b|s|^p}.
$$
Therefore,  by \eqref{theta11}
\begin{equation}\label{theta11_mul}
P\{S_j\cdot\theta>|x|\}\leq e^{- \varkappa_p\left(\frac{|x|}{k}\right)^p k(1+o(1))}.
\end{equation}
For a $d$-dimensional random vector $X$ and arbitrary $\varepsilon>0$ one can find a finite collection of unit vectors $\theta_1,
\ldots, \theta_N, \; N=N(\varepsilon, d)$ such that
$$
\{ |X | > |x| \}  \subset \bigcup_{i=1}^N \{ \theta_i \cdot X >(1-\varepsilon) |x| \}
$$
Then
\begin{equation*}\label{N}
P \{ | X | > |x| \} \le N(\varepsilon, d) \ P \{ \theta \cdot X  > (1-\varepsilon ) |x| \},
\end{equation*}
and together with \eqref{theta11_mul} it gives the desired asymptotic estimate \eqref{theta11-d} for $P (|S_k| > |x|)$ in
the multi-dimensional case.
\end{proof}

We proceed with obtaining point-wise estimates for  $a^{*k} (x)$.
Denote  by $F_k(s)$ the distribution function of $|S_k|$, then in the case $p=1$ we have
$$
a^{\ast (k+1)}(x)  \le  C_1 \int\limits_{0}^{\infty} e^{-b (|x|-s)} \ dF_k (s) =
$$
$$
C_1 \int\limits_{0}^{\frac12 |x| }  e^{-b (|x|-s)} \ dF_k (s) + C_1 \int\limits_{\frac12 |x|}^{\infty} e^{-b (|x|-s)} \ dF_k (s) \le C_1 e^{-\frac12 b |x|} + C_1 P \big\{ |S_k| \ge \frac12 |x| \big\}.
$$
Together with estimates \eqref{theta11-d} and \eqref{P2}, where $\varkappa_1 = \frac b4$, this yields for all $k \le \alpha_1 |x|$:
\begin{equation}\label{akp1}
a^{\ast k}(x) \le 2 C_1 \ e^{- \frac{b}{8} |x|}.
\end{equation}
The case $p > 1$ can be treated similarly, and for any  $k \le \alpha_p |x|$ we obtain
\begin{equation}\label{akpg1}
a^{\ast (k+1)}(x) \le   C_1 \int\limits_{0}^{\infty} e^{-b (|x|-s)^p} \ dF_k (s)
\le C_1 e^{-\frac{b}{2^p} x^p} + C_1 e^{- \varkappa_p \frac{x^p}{2^p k^{p-1}} } \le C_2 e^{- \tilde \varkappa_p \frac{x^p}{k^{p-1}} }.\end{equation}

In order to obtain upper bounds for the terms of the sum in \eqref{v} we make use of 
 estimates  \eqref{akp1}--\eqref{akpg1}. 
 We denote
\begin{equation}\label{def_ss0}
S(z,t) = z \ln t - z \ln z + z + \ln a^{\ast z}(x), \quad S_0 (z,t) = z \ln t - z \ln z + z = z \ln \frac tz +z.
\end{equation}
Notice that
$$
\max_z S_0 (z,t) = S_0(t,t) = t,
$$
and  $S_0(z,t)$ is decreasing in $z$ as $z>t$. Consequently, for any $c>0$ and for sufficiently large $t$ we have
\begin{equation}\label{maxS0}
\max\limits_{z \geq c t^{\frac{\delta+1}{2}}} S_0(z,t) = S_0 \left(c t^{\frac{\delta+1}{2}}, \: t \right) < - \tilde c t^{\frac{\delta+1}{2}} \ln t.
\end{equation}
In the case $p=1$, considering the upper bound $a^{\ast k}(x) \le C_1$, we get
\begin{equation}\label{maxS}
\max\limits_{z \geq c t^{\frac{\delta+1}{2}}} S (z,t) \le \max\limits_{z \geq c t^{\frac{\delta+1}{2}}} S_0(z,t) + \ln C_1  < - \tilde c_1 t^{\frac{\delta+1}{2}} \ln t .
\end{equation}

If $k < \alpha_1| x| = \alpha_1 r t^{\frac{\delta+1}{2}}(1+o(1))$, then
estimate \eqref{akp1} implies the following uniform in $k$ upper bound
\begin{equation}\label{a-astk}
a^{\ast k}(x) < C_2 \ e^{- \frac{b}{8} |x|} =  C_2 \ e^{- \frac{b}{8} r t^{\frac{\delta+1}{2}}(1+o(1))}.
\end{equation}
Consequently,
\begin{equation}\label{maxS1}
\max\limits_{k < \alpha_1 r t^{\frac{\delta+1}{2}}} S (k,t)
\le  S_0(t, t) +  \max\limits_{k < \alpha_1 r  t^{\frac{\delta+1}{2}}} \ln a^{\ast k}(x) \le t - \frac{b}{8} r t^{\frac{\delta+1}{2}}.
\end{equation}

Finally, using \eqref{maxS}  and \eqref{maxS1}, we conclude that in the case $p=1$ the asymptotic estimate \eqref{Pge1} holds with $c_1 = \frac b8 r$. Indeed,
\begin{equation*}\label{peq1.fin}
v(x,t) =  e^{-t}  \sum_{k=1}^{\infty} \frac{ t^k  a^{\ast k} (x)}{k!}  =  e^{-t}  \sum_{k=1}^{\alpha_1 r t^{\frac{\delta+1}{2}} } \frac{ t^k  a^{\ast k} (x)}{k!} + e^{-t}  \sum_{k> \alpha_1 r t^{\frac{\delta+1}{2}}} \frac{ t^k  a^{\ast k} (x)}{k!}
\end{equation*}
$$
\le  \alpha_1 r t^{\frac{\delta+1}{2}}  e^{  - \frac{b}{8} r t^{\frac{\delta+1}{2}}} + O\left( e^{ - t^{\frac{\delta+1}{2}} \ln t} \right) \le
e^{- \frac{b}{8} r t^{\frac{\delta+1}{2}}\ (1+ o(1)) } \quad \mbox{ as } \; t \to \infty.
$$
In the case $p>1$, using estimates \eqref{maxS0} and \eqref{akpg1}, for all $k \le \alpha_p |x| =  \alpha_p r t^{\frac{\delta+1}{2}}(1+o(1))$ we have
$$
\max\limits_{k \le \alpha_p |x|} S(k,t) =  \max\limits_{ k \le \alpha_p |x|} \{ S_0(k,t) + \ln a^{\ast k}(x) \}
 \le  \max\limits_{k \le \alpha_p |x|} \Big\{ S_0 (k,t) - \tilde \varkappa_p r^p \:\frac{t^{\frac{\delta+1}{2} \: p}}{k^{p-1}} \Big\} \le - c^{(1)}_p t^{\frac{\delta+1}{2} } (\ln t)^{\frac{p-1}{p}}.
$$
In order to justify the last inequality we notice that
$$
\hat k(t) := \mathrm{argmax} \, \left\{ S_0 (k,t) - \tilde \varkappa_p r^p \: \frac{t^{\frac{\delta+1}{2} \: p}}{k^{p-1}} \right\} = \hat c \: \frac{t^{\frac{\delta + 1}{2} }}{(\ln t)^{\frac{1}{p}}} = o( t^{\frac{\delta + 1}{2} })
$$
with a constant $\hat c = \hat c(p,\delta ,r)$.
Then $\hat k(t)<\alpha_p|x|$ and, therefore,
$$
\frac{ t^k  a^{\ast k} (x)}{k!}  \le e^{-  c^{(1)}_p t^{\frac{\delta+1}{2} } (\ln t)^{\frac{p-1}{p}}}, \quad k< \hat k(t).
$$
Estimating  $S(k,t)$ for $k > \hat k(t)$ relies on the inequalities  for the function $S_0(k,t)$ similar to those in \eqref{maxS0}  and the upper bound $a^{\ast k}(x) \le C_1$. We have for $ k > \hat k(t)$
$$
\frac{ t^k  a^{\ast k} (x)}{k!} <  C_1 e^{S_0(\hat k(t), t)}  = C_1 e^{\hat k(t) (\ln t - \ln \hat k(t) + 1)} \le e^{-  c^{(2)}_p t^{\frac{\delta+1}{2} } (\ln t)^{\frac{p-1}{p}} (1+ o(1))}.
$$
Finally, taking into account the fact that  $\frac tk<\frac12$ for all $k > \hat k(t) =  \hat c  t^{\frac{\delta+1}{2}}(\ln t)^{- \frac1p}$, we conclude that in the case $p>1$
\begin{equation*}\label{VfinPG1}
\begin{array}{c}
\displaystyle
v(x,t) =  e^{-t}  \sum_{k  \le \hat k(t) } \frac{ t^k  a^{\ast k} (x)}{k!} \ + \
 e^{-t}  \sum_{k > \hat k(t) }  \frac{ t^k  a^{\ast k} (x)}{k!}\\[7mm]
 \displaystyle
\le C_1 t^{\frac{\delta+1}{2}} e^{-t -  c^{(1)}_p t^{\frac{\delta+1}{2} } (\ln t)^{\frac{p-1}{p}}} \ + \  e^{- t -  c^{(2)}_p t^{\frac{\delta+1}{2} } (\ln t)^{\frac{p-1}{p}} (1+o(1))  } \le e^{-  c_p t^{\frac{\delta+1}{2} } (\ln t)^{\frac{p-1}{p}} (1+ o(1)}. 
\end{array}
\end{equation*}
\medskip

Now let us consider the case of $a(x)$ with a compact support. If \eqref{boundedsupp} holds then
\begin{equation}\label{Vfin}
v(x,t) =  e^{-t}  \sum_{k > \frac{r}{\mu} t^{\frac{\delta+1}{2}} } \frac{ t^k  a^{\ast k} (x)}{k!} \le  C_1 e^{-t}  \exp \left\{ \max\limits_{k > \frac{r}{\mu} t^{\frac{\delta+1}{2}}} S_0(k,t) \right\},
\end{equation}
where $S_0$ has been defined in \eqref{def_ss0}.
The function $S_0(k,t)$ is decreasing in variable $k$ as $k>t$, hence
\begin{equation}\label{Vfin1}
\max\limits_{k > \frac{r}{\mu} t^{\frac{\delta+1}{2}}} S_0(k,t) = S_0 \left(\frac{r}{\mu} \: t^{\frac{\delta+1}{2}}, \: t \right) = - \frac{r}{\mu}\: \frac{\delta-1}{2}\: t^{\frac{\delta+1}{2}} \ln t(1+ o(1)).
\end{equation}
Finally, \eqref{Vfin} and \eqref{Vfin1} imply estimate \eqref{Pfinite}.
This completes the proof of Theorem  \ref{LTail-1var3}.

\subsection{The region of large deviations: Proof of Theorem \ref{LTail-2_MD}}

The main step of the proof is obtaining point-wise estimates for $a^{* k}$.

\begin{lemma}\label{LD-bis}
Let  $a(x)$ satisfy  \eqref{a1}--\eqref{a2} and \eqref{lt},  and assume that, in the case  $p = 1$, condition ${\bf A_1}$  holds.
Then for $ x = rt (1+ o(1))$, $r \in R^d \backslash \{0\} $, and for any positive constants $\alpha_1 < \alpha_2$ we have
\begin{equation}\label{akLD-bis1bis}
a^{\ast k} (x) \le e^{-I(\frac{x}{k}) k (1+ o(1))}, \quad \mbox{\rm  if } \quad \alpha_1 t \le k \le \alpha_2 t.
\end{equation}
Furthermore, there exists a positive constant $\alpha_1>0$  such that
\begin{equation}\label{akLD-bis3bis}
a^{\ast k} (x) \le e^{- I(r) \, t(1+o(1)) } ; \quad \mbox{\rm if } \; 1 \le k \le \alpha_1 t.
\end{equation}
If condition \eqref{boundedsupp}  is fulfilled,  then for any $\alpha_2$
\begin{equation}\label{akLD-bis4bis}
a^{\ast k} (x)=0, \; \mbox{\rm if } \; k\leq\frac{|x|}{\mu},\quad \mbox{\rm and } \quad
a^{\ast k} (x) \le e^{-I(\frac{x}{k}) k (1 + o(1))}, \; \mbox{\rm if } \;  \frac{ |x|}{\mu} \le k \le \alpha_2 t.
\end{equation}
\end{lemma}
%
%
\begin{proof}

We start with the case $p=1$, $ \alpha_1 t \le k \le \alpha_2 t$. The kernel $a^{\ast (k+1)}(x)$  can be written as follows:
\begin{equation}\label{three-bis}
a^{\ast (k+1)}(x) = \int\limits_{\{ z: \ I(\frac{z}{k}) \le I(\frac{x}{k}) \}} a^{\ast k} (z) a(x-z) dz + \int\limits_{\{ z: \ I(\frac{z}{k}) > I(\frac{x}{k}) \}} a^{\ast k} (z) a(x-z) dz.
\end{equation}
Denote by $A_1 = \{ z: \ I(\frac{z}{k}) \le I(\frac{x}{k}) \}, \; A_2 =\{ z: \ I(\frac{z}{k}) > I(\frac{x}{k}) \}$.
Using the large deviations principle for the sum of i.i.d. random vectors, see \cite{DZ}, we obtain the upper estimate for the second integral in \eqref{three-bis}, when $ \alpha_1 t \le k \le \alpha_2 t$:
\begin{equation}\label{first}
\int\limits_{A_2} a^{\ast k} (z) a(x-z) dz \le C_1 P (S_k \in A_2) \le C_1 \exp{\left\{- \inf_{k\rho \in A_2} I(\rho) \, k +o(k)\right\}} = e^{-I(\frac{x}{k}) k+o(k)}.
\end{equation}
To estimate the first integral in \eqref{three-bis} we introduce
$$
F_k (s) = \int\limits_{\{z: \ I(\frac{z}{k}) k \ge s  \}} a^{\ast k}(z) dz, \qquad \mbox{where } \; s \in \left( 0, I(\frac xk) k \right),
$$
then we have
\begin{equation}\label{Fbis}
F_k (0) =1, \quad F_k (\infty)=0, \quad F_k (s) = e^{-s (1+o(1))} \; (k \to \infty).
\end{equation}
If we denote $L_s = \{ z: \  I(\frac zk) k = s \}$, then ${\rm dist} (x , L_s)$ is a decreasing continuous function on $[0, I(\frac xk) k]$,
it is smooth on $(0,I(\frac xk) k]$.  For each $s\in [0, I(\frac xk) k]$ there exists a unique $z_s \in L_s$ such that
${\rm dist} (x , L_s)= |x - z_s|$. All these assertions are elementary consequences of  convexity of the function $I$. Clearly, $z_0 =0, \ z_{ I(\frac xk) k } = x$.

It follows from Proposition \ref{I_MD} and estimate \eqref{lt} that for any $x,z \in R^d$, such that $I(\frac{z}{k}) < I(\frac{x}{k})$, the following inequality holds true:
\begin{equation*}
\left(  I(\frac{x}{k}) - I(\frac{z}{k}) \right) \, k \le \max_{r \in l(\frac xk, \frac zk)} |\nabla I(r) | \, |x-z| \le b \, |x-z| \le - \ln a(x-z),
\end{equation*}
where by $ l(x, y)$ we denote the segment connecting points $x$ and $y$. Consequently,
\begin{equation}\label{Ixz}
e^{- I(\frac zk) k } \, a(x-z) \le e^{- I(\frac xk) k }.
\end{equation}
Then using \eqref{Fbis} and inequality \eqref{Ixz} we rewrite the first integral in \eqref{three-bis} as follows:
$$
\int\limits_{A_1} a^{\ast k} (z) a(x-z) dz  \le C_1 \int\limits_{A_1} a^{\ast k} (z) e^{-b|x-z|} dz  \le  C_1 \int\limits_0^{I(\frac{x}{k}) k} e^{-b \, {\rm dist} (x, L_s)} d(-F_k(s))
$$
$$
 = C_1 e^{-b \, {\rm dist} (x, L_s)} F_k(s) \Big|_{I(\frac{x}{k}) k}^{0} \ - \
C_1 b \int\limits_0^{I(\frac{x}{k}) k} F_k(s) e^{-b \, {\rm dist} (x, L_s)}\, \frac{d}{ds} {\rm dist} (x, L_s) ds
$$
$$
= C_1 e^{-b|x|} - C_1 e^{- I(\frac{x}{k}) k (1+o(1))} \ - \
C_1 b \int\limits_0^{I(\frac{x}{k}) k} e^{-s (1+ o(1))} e^{-b |x - z_s|} \, \frac{d}{ds} {\rm dist} (x, L_s) ds
$$
$$
\le  e^{-I(\frac{x}{k}) k (1+ o(1))}  \ - \
C_1 b \int\limits_0^{I(\frac{x}{k}) k} e^{- I(\frac{z_s}{k}) k (1+ o(1))} e^{-b |x - z_s|} \, \frac{d}{ds} {\rm dist} (x, L_s) ds
$$
$$
\le  e^{-I(\frac{x}{k}) k (1+ o(1))} - C_1 b  e^{-I(\frac{x}{k}) k (1+o(1))} \int\limits_0^{I(\frac{x}{k}) k} \frac{d}{ds} {\rm dist} (x, L_s) ds
$$
$$
 \le  e^{-I(\frac{x}{k}) k (1+ o(1))} + C_1 b  e^{-I(\frac{x}{k}) k (1+o(1))} |x| \le  e^{-I(\frac{x}{k}) k (1+ o(1))}.
$$
This inequality together with \eqref{first} imply \eqref{akLD-bis1bis} in the case $p=1$.

\medskip

To prove the upper bound \eqref{akLD-bis1bis} for $p>1$ and  $ \alpha_1 t \le k \le \alpha_2 t $ we rewrite $a^{\ast (k+1)}(x)$ as a sum
\begin{equation}\label{14A}
a^{\ast (k+1)}(x) =  \int\limits_{|z-x| <  h k^{1/p}} a^{\ast k} (z) a(x-z) dz +
\int\limits_{|z-x| \ge h k^{1/p}} a^{\ast k} (z) a(x-z) dz.
\end{equation}
The second integral in \eqref{14A} has an upper bound
$$
\max\limits_{|u| \ge  h k^{\frac1p}} a(u) \le C_1 e^{-b h^p k}.
$$
If the constant $h>0$ is taken in such a way that $b h^p > I(\frac{r}{\alpha_1})$, then $b h^p > I(\frac{r}{\alpha_1}) > I(\frac{x}{k})$ for any  $k \in [\alpha_1 t, \alpha_2 t] $. Thus, the second term in \eqref{14A} is bounded by \eqref{akLD-bis1bis}.

For $k \sim t$ and for arbitrary $\varkappa>0$ the first term in \eqref{14A} can be estimated from above as
$$
\int\limits_{|z-x| <  h k^{1/p}} a^{\ast k} (z) a(x-z) dz \le C_1 \int\limits_{|z-x| <  h k^{1/p}} a^{\ast k} (z) dz \le C_1  \Pr \{ |S_k -x| < \varkappa k \}
$$
$$
\le C_1 \exp \{ - \inf_{\rho \in A_{\varkappa}} I(\rho) k + o(k)  \} \le  e^{- I(\frac{x}{k}) k + o(k)},
$$
where $A_{\varkappa} = \{z: \ |z-\frac{x}{k}| < \varkappa \}$. Here we used the large deviations principle for estimating $ \Pr \{ |S_k -x| < \varkappa k \}$  and continuity of $I(r)$.

\medskip
In the case $p=1$ and $k \le \alpha_1 t$ we apply the upper bound
\begin{equation}\label{p1ksmall}
a(x) \le A \ \tilde a (|x|), \qquad \mbox{ with some constant } \; \quad A>1,
\end{equation}
where $\tilde a(|x|) = a_1 e^{-b |x|}$ is a spherically symmetric kernel satisfying \eqref{lt} with the same $b$.
Next we need the following statement for 1-D random variables.

\begin{proposition}\label{LD-proposition}
Let  $a(x), x \in R,$ satisfy \eqref{lt} with $p=1$ and condition ${\bf A_1}$ holds.
Then there exists positive constant $\tilde C_1$ such that \\
\begin{equation}\label{akLD}
a^{\ast k} (x) \le \tilde C_1 e^{- I(\frac{x}{k}) k (1+o(1)) }; \quad \mbox{\rm for all } \; k \ge 1.
\end{equation}

\end{proposition}

\begin{proof}
We represent $a^{\ast (k+1)}(x), \, x \in \mathbb{R}$ as follows:
\begin{equation}\label{three}
a^{\ast (k+1)}(x) = \int\limits_{-\infty}^0 a^{\ast k} (z) a(x-z) dz + \int_{0}^{x} a^{\ast k} (z) a(x-z) dz +
\int_{x}^{\infty} a^{\ast k} (z) a(x-z) dz.
\end{equation}
Since  $I(\frac{x}{k}) k < bx$, the first integral in \eqref{three} admits the estimate
$$
\int\limits_{-\infty}^0 a^{\ast k} (z) a(x-z) dz  \le C_1 e^{-b x}  \le e^{-I(\frac{x}{k}) k}.
$$
For the last integral in \eqref{three} we apply the Markov inequality:
\begin{equation}\label{A}
\int\limits_{x}^{\infty} a^{\ast k} (z) dz = P(S_k> x)   \le \inf\limits_\gamma \frac{\left( \mathbb{E} e^{\gamma X_1} \right)^k}{ e^{k \gamma \frac{x}{k}}} = \inf\limits_\gamma e^{k L(\gamma) - k \gamma  \frac{x}{k}} = e^{- I(\frac{x}{k}) k }.
\end{equation}
Then  we get for any $k$ and any $x>0$
$$
\int\limits_{x}^{\infty} a^{\ast k} (z) a(x-z) dz  \le C_1 P (S_k > x)  \le C_1 e^{-I(\frac{x}{k}) k }.
$$
To estimate the second integral  in \eqref{three} denote $\tilde F_k(x) = \int\limits_x^{\infty}  a^{\ast k} (z) dz = P (S_k > x)$. Then
$$
\int\limits_0^{x} a^{\ast k} (z) a(x-z) dz  \le C_1  \int\limits_0^{x} e^{-b(x-z)} d(-\tilde F_k (z)) =
 C_1  e^{-b(x-z)} \tilde F_k (z)|^0_x +  C_1 b \int\limits_0^{x} e^{-b(x-z)} \tilde F_k (z) dz \le
$$
$$
C_1  e^{-b x} +  C_1 b  \int\limits_0^{x} e^{- I(\frac{z}{k}) k - b(x-z)} dz \le C_1  e^{-b x} +  C_1 b x  e^{- I(\frac{x}{k}) k } \le  C_2 x e^{- I(\frac{x}{k}) k };
$$
we have used here the inequalities
$$
I(\frac{x}{k})  - I(\frac{z}{k}) < b \, \frac{x-z}{k} \quad \mbox{ for all } \; z \in (0,x), \quad \mbox{ and } \; \; I(\frac{x}{k}) k < bx.
$$
Considering $x =rt(1+ o(1))$ we obtain  estimate \eqref{akLD}  for all $k\geq 1$.
\end{proof}

Then using \eqref{akLD} we have
$$
a^{\ast k} (x) \le A^k \tilde a^{\ast k} (|x|) \le A^k  e^{- I_{\tilde a}(\frac{|x|}{k}) k (1+o(1)) }.
$$
Since $ k < \alpha_1 t$ with a small $\alpha_1$, then $\frac{|x|}{k} > \frac{|r|}{\alpha_1} \gg 1$,
and using asymptotic representation \eqref{Iinfty}  for  $I_{\tilde a}(s)$ as $s \to \infty$ and inequality \eqref{I(r)P6}, we conclude
that for any $\delta>0$ there exists $\alpha_1>0$ such that
\begin{equation}\label{smallk}
a^{\ast k} (x) \le A^k \tilde a^{\ast k} (|x|) \le A^k  e^{-b |x| (1-\delta)} = A^k e^{-b|r| t (1-\delta)} \le
e^{-b|r| t (1-\delta)+\alpha_1t\ln A}\leq
 e^{-I(r) t}.
\end{equation}
In order to obtain the last inequality we chose $\delta = \frac{(b|r|-I(r))}{4 b|r|}$ . Thus \eqref{akLD-bis3bis} is proved for $p=1$.

\medskip
If $p>1$ and $k \le \alpha_1 t$, then for sufficiently small $ \alpha_1$ recalling that $x =rt(1+o(1))$,
from the Markov inequality \eqref{A}  we have
\begin{equation}\label{4bis}
P \{ |S_k|  > \frac12 |x| \} \le   e^{- \tilde I(\frac{|x|}{2 k}) k } \le   e^{- \tilde I(\frac{|r|}{2 \alpha_1}) \alpha_1 t }  \le   e^{- 2 I(r) t },
\end{equation}
where $\tilde I(s)$ is the rate function for the 1-D random variable $|X|$.
Here we used the fact that the function $J(\alpha) = \alpha \tilde I(\frac{s}{\alpha})$ is decreasing in  $\alpha \in (0,1]$, that is a consequence of convexity of $\tilde I(s)$. Moreover, by \eqref{gti} we have $J(\alpha) \to \infty$ as $\alpha \to 0+$. Then using \eqref{4bis} we conclude that for a small enough constant $\alpha_1>0$ we get
\begin{equation*}\label{akpg1_bbis}
a^{\ast (k+1)}(x)  = \int\limits_{|z| \le \frac12 |x| } a^{\ast k}(z) \ a(x-z) \ dz + \int\limits_{|z|>\frac12 |x|} a^{\ast k}(z) \ a(x-z) \ dz
\end{equation*}
$$
\le  C_1 e^{- b \left(\frac{|x|}{2} \right)^p } + C_1 e^{- \tilde I(\frac{|r|}{2 \alpha_1}) \alpha_1 t }  \le   C_1 e^{- \tilde b \left(\frac{|r|}{2} \right)^p t^p} + C_1 e^{- \tilde I(\frac{|r|}{2 \alpha_1}) \alpha_1 t }  \le  \tilde C_2 e^{- I(r) t }.
$$
This completes the proof of estimates \eqref{akLD-bis1bis} - \eqref{akLD-bis3bis}.

The first relation in  \eqref{akLD-bis4bis} is evident. The proof of the second one is based on the same arguments
as those used in the case $p>1$.
\end{proof}

\medskip
Combining Stirling's formula with the estimates of Lemma \ref{LD-bis} we obtain the following statement.

\begin{corollary}\label{coro_1}
Let the assumptions of Lemma \ref{LD-bis} be fulfilled.
If $x = r t (1+ o(1))$, then for all $k$ such that $\alpha_1 t \le k \le \alpha_2 t$ with arbitrary positive numbers $\alpha_2$ and $\alpha_1$, estimate \eqref{akLD-bis1bis} implies that
\begin{equation}\label{Ir}
\frac{t^k a^{\ast k}(x)}{k!} \le \exp \left\{ k\ln t - k \ln k + k - I( \frac{x}{k}) k +o(t) \right\} = \exp \{S(k,t)+o(t)\},  \quad t \to \infty,
\end{equation}
where $S(k,t) = k\ln t - k \ln k + k - I( \frac{rt}{k}) k$.

\end{corollary}


Recalling the definition of $\xi_r$ in \eqref{eqeq} and the function $\Phi$ in \eqref{Phi(r)} we have
\begin{equation}\label{Phi}
e^{-t + S(\hat z,t)} = \exp \left\{ t \left( -1 + \frac{1}{\xi_r}(1 +  \ln \xi_r - I(\xi_r r )) \right)\right\} = \exp \{ -
\Phi(r) t\},
\end{equation}
where $\hat z = \mathrm{argmax} S(z,t)$.
If $x=rt(1+ o(1))$ as $t \to \infty$, then the following upper bound
\begin{equation}\label{UpperB1}
e^{-t} \, \frac{t^k a^{\ast k}(x)}{k!} \ \le \ e^{ - \Phi(r) t (1 + o(1)) }
\end{equation}
is valid for all $k$ from the interval $k \in (\alpha_1 t, \alpha_2 t)$.

\medskip

To  estimate  $v(x,t)$  from above we decompose the sum in (\ref{v}) into three parts:
\begin{equation}\label{D1v}
v(x,t) \ =  \ e^{-t}  \sum_{k<\alpha_1 t} \frac{ t^k  a^{\ast k} (x)}{k!} +  e^{-t}  \sum_{k=\alpha_1 t }^{\alpha_2 t} \frac{ t^k
a^{\ast k} (x)}{k!} +  e^{-t}  \sum_{k > \alpha_2 t} \frac{ t^k  a^{\ast k} (x)}{k!}.
\end{equation}
For the first sum in \eqref{D1v} we
apply upper bound \eqref{akLD-bis3bis}.  This together with \eqref{Phi} yield
\begin{equation}\label{127A}
a^{\ast k} (x) \le  e^{- I(r) t +o(t) }  =  e^{-t+S(t,t)+o(t)} \le e^{-t+S(\hat z,t)+o(t)} =  e^{-\Phi(r) t (1+ o(1)) },
\end{equation}
because $- I(r) t = S(t,t) - t$, and $\max_z S(z,t) = S(\hat z, t)$ with $\hat z >t$.
Consequently,
\begin{equation}\label{D1v.1}
e^{-t} \ \sum_{k=1}^{\alpha_1 t} \frac{ t^k  a^{\ast k} (x)}{k!}  \le   e^{ - \Phi (r) t (1+ o(1)) }, \qquad\mbox{as }t \to \infty.
\end{equation}
For the third sum, if $k > \alpha_2 t$ with $\alpha_2>2$ then  we have
$$
\frac{t^k}{k!} < \frac{t^{\alpha_2 t}}{(\alpha_2 t)!}< e^{(\alpha_2 - \alpha_2 \ln \alpha_2) t}.
$$
Choosing $\alpha_2>2$ such that $1 - \alpha_2 + \alpha_2 \ln \alpha_2 > \Phi (r)$, we obtain
\begin{equation}\label{D1v.3}
e^{-t} \ \sum_{k>2t} \frac{ t^k  a^{\ast k} (x)}{k!} \le  C_1   e^{(-1 + \alpha_2 - \alpha_2 \ln \alpha_2  ) t} < e^{- \Phi(r) t}.
\end{equation}

It remains to estimate the second sum on the right-hand side of \eqref{D1v}.
To this end we use \eqref{UpperB1}, then
\begin{equation}\label{D1v.2}
e^{-t} \ \sum_{k=\alpha_1 t}^{\alpha_2 t} \frac{ t^k  a^{\ast k} (x)}{k!} \le \alpha_2 t e^{ -\Phi (r) t (1 + o(1))} = e^{ -\Phi(r)
 t (1 + o(1))}, \qquad t \to \infty.
\end{equation}
Finally, in the region $ x=rt (1+o(1)), \; r\not=0$,  from \eqref{D1v.1} - \eqref{D1v.2} we deduce:
\begin{equation}\label{D1v.fin}
v(x,t) \le e^{ -\Phi (r) t (1 + o(1))}, \qquad t \to \infty.
\end{equation}

For $a(x)$ with a finite support we take $\alpha_1=r/\mu$ in \eqref{D1v}. Then the first sum on the right-hand side of \eqref{D1v} does not contribute. Estimating the two other sums relies on  \eqref{UpperB1}, \eqref{D1v.3} and \eqref{D1v.2} like in the case
$p>1$.
This completes the proof of \eqref{AsympMD}.

\medskip

It remains to show that the function $\Phi(r)$ satisfies the asymptotic relations in \eqref{Phi0-bis2}--\eqref{Phibs-bis}.
Considering the properties of the function $I(r)$, in particular \eqref{I0-bis}, it is easy to see
that
$ \xi_r  = 1 - \frac{r^2}{2 \sigma}+ o(r^2)$, as $r\to 0$. Recalling now the definition of $\Phi (r)$ in \eqref{Phi(r)}, we finally obtain asymptotic formula \eqref{Phi0-bis2}.
\medskip

The asymptotics of  $\Phi (r)$ for large $r$ depends crucially on the rate of decay  of  $a(x)$ at infinity.
We start with the case, when $a(x)$ satisfies \eqref{boundedsupp}. Then from Proposition \ref{I_MD} it follows that $I(\xi r) = \infty$ for all $\xi |r| > \mu$. Then the solution $\xi_r$  of equation \eqref{eqeq} satisfies the inequality $\xi_r < \frac{\mu}{|r|}$.
By the definition of $\Phi(r)$ we have
$$
\Phi (r) = 1 - \frac{1}{ \xi_r} +  \frac{1}{ \xi_r } \ln \frac{1}{ \xi_r } + \frac{1}{\xi_r } I(r \xi_r) \ge \frac{1}{ \xi_r }\Big( \ln \frac{1}{ \xi_r } - 1 \Big),
$$
Therefore, for large enough $r$,
$$
\Phi (r)  \ge \min\limits_{x \in (\frac{|r|}{\mu}, \infty)} x (\ln x-1) = \frac{|r|}{\mu} \, \Big(\ln   \frac{|r|}{\mu}-1\Big),
$$
and we obtain \eqref{Phibs-bis}.
\medskip

Since the principal term on the right-hand side of  \eqref{Iinfty-bisbis} only depends on $|r|$ as $r\to\infty$, then  in the case $p > 1$ for the solution $\xi_r$ of equation \eqref{eqeq}
we have  $\xi_r = \xi_{|r|}(1+o(1))$, as $r\to\infty$.  Therefore, we can reduce the general case to the spherically symmetric case
(or the 1-D case). Notice that for any $p \ge 1$ condition ${\bf A_p}$ implies \eqref{N0}.
The next statement describes the asymptotic behaviour of $\xi_r$ for large $r$ under the assumption that  \eqref{lt} and \eqref{N0} hold true.

\begin{proposition}\label{xi}
Let  \eqref{lt} and \eqref{N0} hold. Then
\begin{equation}\label{ur}
u(s) := s \, \xi_s  \to \infty, \quad \mbox{ as } \; s \to \infty.
\end{equation}
If $p>1$ and condition $ {\bf A_p}$ is fulfilled, then
\begin{equation}\label{hatxi}
\xi_s =  h_p \frac{(\ln s)^{1/p}}{s}  (1+o(1)), \quad \mbox{as} \quad  s \to \infty,
\end{equation}
where $ h_p$ is a constant depending on $p$ and $b$.
\end{proposition}

\begin{proof}
We first prove \eqref{ur}. If we assume that  $u(s)$ is bounded: $u = u(s) < a$, then for all $s>0$ the function $\ln \xi_s$ is bounded from below:
\begin{equation}\label{ur1}
\ln \xi_s = I(u) - u I'(u)> I(a) - a I'(a) > - \infty.
\end{equation}
We have used here the facts that  $J(u) = I(u) - u I'(u)$ is a decreasing function on $[0,+\infty)$, and due to condition \eqref{N0} the functions $I(u), I'(u)$ are finite for all $u>0$.
On the other hand,
$$
\ln \xi_s = \ln \frac{u(s)}{s} < \ln a - \ln s.
$$
For large $s$ this inequality  contradicts \eqref{ur1}. This proves \eqref{ur}.

The function $J(u) = I(u) - u I'(u)< 0$ is negative for all $u>0$, because $J(0)=0$, $J'(u) \le 0$ for   $u \ge 0$, and $J'(u)<0$ for $0 \le u<\kappa_0$ with some $\kappa_0>0$. In the case  $p>1$ combining this inequality with   \eqref{eqeq}, \eqref{Iinfty}  and \eqref{ur} we conclude that
$$
\ln \frac{1}{ \xi_s} \ =\ b(p-1)\left( r  \xi_s \right)^p(1+o(1)), \quad s \to \infty.
$$
Consequently, we get \eqref{hatxi} in the case $p>1$: $ \xi_s =   h_p \frac{(\ln s)^{1/p}}{s}  (1+o(1))$ with $h_p=\big(b(p-1)\big)^{-1/p}$.
\end{proof}

Inserting  \eqref{hatxi} into \eqref{Phi(r)}, we finally obtain asymptotic formulas \eqref{Phi-inftyG1-bis2} and \eqref{Phi-inftyG1}.

\medskip

In the case $p=1$ using \eqref{eqeq} and \eqref{Iinfty-bis} for large $r$ we get
$$
\Phi (r) = 1 - \frac{1}{ \xi_r}\Big(1 +  \ln \xi_r - I(\xi_r r) \Big) = 1 - \frac{1}{ \xi_r } + b|r| (1+o(1)).
$$
According to \eqref{ur} we have $u(r) = |r| \xi_r \to \infty$ as $|r| \to \infty$, consequently, $\frac{1}{ |r| \xi_r} \to 0$, and $ \frac{1}{ \xi_r} = o(|r|)$. Thus,
$$
\Phi (r) =  b \,|r|\, (1+o(1)) \quad \mbox{ as } \; |r| \to \infty,
$$
and asymptotic formula \eqref{Phi-infty1-bis2} is proved. Theorem \ref{LTail-2_MD} is completely proved.

\subsection{The region of large deviations. Proof of Theorem \ref{Asymptotics}}

In order to justify the asymptotics in  \eqref{AsympMD-bis} it suffices to prove that for $x=rt(1+o(1))$ we have
\begin{equation}\label{dopest1}
 e^{- \Phi(r) t (1+ \nu_1(t))}\leq v(x,t) \le e^{- \Phi(r) t (1+ \nu_2(t))},
\end{equation}
where $\nu_j (t) \to 0$ as $t\to\infty$, $j=1,2$. Since the upper bound has already been proved, see \eqref{AsympMD},
we proceed with the lower bound. Denote $\hat r=x/t$. Then $\hat r=r(1+o(1))$ as $t\to\infty$.

From the definition of
$\xi_r$ in \eqref{eqeq} by the implicit function theorem we obtain that $\xi_r$ is a smooth function of $r$. So is
$r\xi_r$. Letting $r_0^*=\xi_r r$ and $r^*=\xi_{\hat r} \hat r$, we then have $r^*=r_0^*(1+o(1))$.

We define $\gamma^*(r^*)\in\mathbb R^d$ as a solution to the equation $\nabla L(\gamma)=r^*$. By Proposition \ref{DR}
this equation has a unique solution. Moreover, $\gamma^*(r^*)$ is a smooth function of $r^*$. In particular,
  $\gamma^*_0=\gamma^*(r_0^*)=\gamma^*(r^*)(1+o(1))$, as $t\to\infty$.  We recall, see Proposition \ref{DR} again, that for a random variable
  $X_{\gamma^*}$ with the density
$a_{\gamma^*}(x) = \frac{a(x) e^{\gamma^* \cdot x}}{\Lambda(\gamma^*)}$ its expectation is equal to  $r^*$.

Consider a family of densities $\widetilde a_{\gamma^*}(x)=a_{\gamma^ *}(x+r^*)$ and the corresponding random variables $\widetilde X_{\gamma^*}=X_{\gamma^*}-r^*$.
\begin{lemma}\label{l_property_star}
There exists a neighbourhood $\mathcal{O}$ of $\gamma^*_0$ in $\mathbb R^d$ such that for all $\gamma^*\in \mathcal{O}$
the density $\widetilde a_{\gamma^*}$ possesses the following properties:
\begin{itemize}
\item [\bf a.] \ \ $\widetilde a_{\gamma^*}(x)\leq Ce^{-\mu |x|}$ \ for some $\mu>0$ and $C>0$.
\item [\bf b.] \ \  The matrix
$$
\sigma_{ij}(\gamma^*)=\int_{\mathbb R^d}x_ix_j\widetilde a_{\gamma^*}(x)\,dx
$$
is positive definite, $\sigma(\gamma^*)\zeta\cdot\zeta\geq\mu_1|\zeta|^2$ for some $\mu_1>0$ and for all $\zeta\in\mathbb R^d$.
\end{itemize}
The constants $\mu$, $\mu_1$ and $C$ do not depend on the choice of $\gamma^*\in\mathcal{O}$.
\end{lemma}
\begin{proof}
If $p=1$ then under condition ${\bf A}_1^s$ we have $|\gamma^*_0|<b$.   We can choose sufficiently small neighbourhood
$\mathcal{O}$
of $\gamma^*_0$ in such a way that the inequality $b-|\gamma^*|\geq \frac12(b-|\gamma^*_0)|$ holds for all $\gamma^*\in\mathcal{O}$.
 It is clear that $\Lambda(\gamma^*)\geq C>0$ for all $\gamma^*\in\mathbb R^d$.
This implies, in view of \eqref{lt} and the definition of $\widetilde a_{\gamma^*}$,  the first statement of Lemma with $\mu_1=\frac12(b-|\gamma^*_0|)$. If $p>1$, then this statement is obvious.

The second statement of Lemma is a straightforward consequence of the first one. Indeed, it follows from ${\bf a.}$ that
there exists $R_0>0$ such that
$$
\int\limits_{Q_{R_0}}\widetilde a_{\gamma^*}(x)dx\geq \frac12
$$
for all $\gamma^*\in\mathcal{O}$, here $Q_{R_0}$ stands for the ball of radius $R_0$ centered at the origin.
Then for any $\theta \in S^{d-1}$  we have
$$
\sigma_{ij}(\gamma^*)\theta_i\cdot\theta_j=\int\limits_{\mathbb R^d}(x\cdot\theta)^2\widetilde a_{\gamma^*}(x)\,dx
\geq \int\limits_{Q_{R_0}\setminus \Pi_\delta}(x\cdot\theta)^2\widetilde a_{\gamma^*}(x)\,dx,
$$
where $\Pi_\delta=\{x\in Q_{R_0}\,:\,|x\cdot\theta|<\delta\}$. Due to ${\bf a.}$ there exists  $\delta_0>0$ such that
 $\int_{\Pi_{\delta_0}}\widetilde a_{\gamma^*}(x)\,dx\leq \frac14$ for all $\gamma^*$ and for all
$\theta\in S^{d-1}$.  Therefore,
$$
\int\limits_{Q_{R_0}\setminus \Pi_{\delta_0}}(x\cdot\theta)^2\widetilde a_{\gamma^*}(x)\,dx\geq \frac14\delta_0^2.
$$
This yields ${\bf b.}$
\end{proof}
It follows from Lemma \ref{l_property_star} that the local limit theorem applies to a family of i.i.d. random variables with the density $\widetilde a_{\gamma^*}$, see Theorems 19.1 and  19.2 in \cite{BhaRao}.   Therefore,
$$
\widetilde a_{\gamma^*}^{\ast k}(0)=(2\pi k)^{-\frac d2}|\sigma(\gamma^*)|^{-1}\big(1+o(1)\big),
$$
as $k\to\infty$, and
\begin{equation}\label{lltrel}
 a_{\gamma^*}^{\ast k}(kr^*)=(2\pi k)^{-\frac d2}|\sigma(\gamma^*)|^{-1}\big(1+o(1)\big).
\end{equation}
 Moreover, by Theorem 19.2 in \cite{BhaRao}, the convergence is uniform in $\gamma^*\in \mathcal{O}$.

\bigskip
According to \eqref{N3},
$$
a^{\ast k}(k r^*) = a^{\ast k}_{\gamma^*}(k r^*) e^{-I(r^*) k}.
$$
Take $k=\left[\frac t\xi_{\hat r}\right]$, where $[\cdot]$ stands for the integer part. Then $kr^*=rt(1+o(1))=x(1+o(1))$,
as $t\to\infty$. Considering \eqref{Phi(r)},
\eqref{Phi} and \eqref{lltrel} and the fact that the convergence in \eqref{lltrel} is uniform in $\gamma^*\in \mathcal{O}$, we conclude that, under this choice of $k$,
$$
\frac{t^k}{k!}a^{\ast k}(k r^*) =a^{\ast k}_{\gamma^*}(k r^*)e^{-\Phi(\hat r)t(1+o(1))}
=e^{-\Phi(r)t(1+o(1))}.
$$
This yields the desired lower bound in \eqref{dopest1}.

{\bf Acknowledgements.} The authors are grateful to Prof. M. Lifshits for fruitful discussions on large deviation principle. The third and the fourth authors would like to thank the Department of Mathematics
of the University of Bielefeld for hospitality.

\end{document}